\documentclass[11pt]{amsart}
\usepackage{amssymb}
\usepackage{amsmath}
\usepackage{amsfonts,latexsym,amstext}

\usepackage{xcolor}
\usepackage [T1]{fontenc}
\usepackage [latin1]{inputenc}
\newtheorem{theorem}{Theorem}[section]
\newtheorem{definition}[theorem]{Definition}

\newtheorem{example}[theorem]{Example}
\newtheorem{lemma}[theorem]{Lemma}
\newtheorem{proposition}[theorem]{Proposition}
\newtheorem{corollary}[theorem]{Corollary}
\newtheorem{remark}[theorem]{Remark}

\begin{document}
\parskip .2cm
%\begin{frontmatter}

%\title{$M$-ideals of homogeneous polynomials: the vector-valued case}

\title{$M$-structures in vector-valued polynomial spaces}

\author{Ver\'{o}nica Dimant}

\author{Silvia Lassalle}

\thanks{Partially supported by CONICET-PIP 11220090100624. The second author was also partially supported by UBACyT X218 and UBACyT X038.}

\address{Departamento de Matem\'{a}tica, Universidad de San
Andr\'{e}s, Vito Dumas 284, (B1644BID) Victoria, Buenos Aires,
Argentina and CONICET.} \email{vero@udesa.edu.ar}

\address{Departamento de Matem\'{a}tica - Pab I,
Facultad de Cs. Exactas y Naturales, Universidad de Buenos Aires,
(1428) Buenos Aires, Argentina and CONICET.} \email{slassall@dm.uba.ar}


\begin{abstract}
This paper is concerned with the study of $M$-structures in spaces of polynomials.
More precisely, we discuss for $E$ and $F$ Banach spaces, whether the class of weakly continuous  on bounded sets $n$-homogeneous polynomials, $\mathcal P_w(^n E, F)$, is an $M$-ideal in the space of continuous $n$-homogeneous polynomials $\mathcal P(^n E, F)$. We show that there is some hope for this to happen only for a finite range of values of $n$. We establish sufficient conditions under which the problem has positive and negative answers and use the obtained results to study the particular cases when $E=\ell_p$ and $F=\ell_q$ or $F$ is a Lorentz sequence space $d(w,q)$. We extend to our setting the notion of property $(M)$ introduced by Kalton  which allows us to lift $M$-structures from the linear to the vector-valued polynomial context.
Also, when $\mathcal P_w(^n E, F)$ is an $M$-ideal in $\mathcal P(^n E, F)$ we prove a Bishop-Phelps type result for vector-valued polynomials and relate norm-attaining polynomials with farthest points and remotal sets.
\end{abstract}

\maketitle

\section*{Introduction}
$M$-ideals emerged in the geometric theory of Banach spaces as a generalization, to the Banach space setting, of the closed two-sided ideals in a $C^*$-algebra. This notion, introduced by Alfsen and Effros in their seminal article \cite{AE} of 1972, leads us to a better understanding of the isometric structure of a Banach space in terms of geometric and analytic properties of the closed unit ball of the dual space. To be more precise, a closed subspace $J$ of a Banach space $X$ is an {\bf $M$-ideal} in $X$, if its annihilator, $J^\perp$, is the kernel of a projection $P$ on the dual space $X^*$ such that $\|x^*\|=\|P(x^*)\|+\|x^*-P(x^*)\|$, for all $x^*\in X^*$. When $J$ is an $M$-ideal in $X$, the canonical complement of $J^\perp$ in $X^*$ is (isometrically) identified with $J^*$. Then, we may write $X^*=J^\perp \oplus_1 J^*$, which in some sense tells us that there is a maximum norm structure underlying the geometry of the unit ball of $X$ and this structure is closely related to $J$. If it is possible to decompose $X$ as $J\oplus_\infty \widetilde J$, for some closed subspace $\widetilde J$ of $X$, we say that $J$ is an {\bf $M$-summand} of $X$. Clearly, $M$-summands are $M$-ideals, but there exist subtle differences. For instance, $c_0$ is an $M$-ideal in $\ell_\infty$ and it is not an $M$-summand. Since $M$-ideals appeared, they have been intensively studied. A comprehensive exposition of the main developments in this subject can be found in the outstanding book by  Hardmand, Werner and Werner \cite{HWW}.

The Gelfand-Naimark theorem states that any arbitrary $C^*$-algebra is isometrically $*$-isomorphic to a $C^*$-algebra of bounded operators on a Hilbert space. Here the only norm closed two-sided $*$-ideal is the subspace of compact operators. Then, it is natural to investigate under which conditions the closed subspace $J$ of compact operators between Banach spaces $E$ and $F$, $J=\mathcal K(E,F)$, results an $M$-ideal in $X=\mathcal L(E,F)$, the space of linear and bounded operators, endowed with the supremum norm. During the last thirty years a number of papers have been devoted to this question (see, for example \cite{HWW,H,K,KW,L,Oja-91,OW}), where the case $E=F$ is of special interest.

In this paper we focus our study in determining the presence of an $M$-structure in the space of continuous $n$-homogeneous polynomials between Banach spaces $E$ and $F$, denoted by $\mathcal P(^n E, F)$. Here the lack of linearity and, more specifically, the degree of homogeneity will play a crucial role. In the polynomial setting, the space of {\it compact operators} is usually replaced by the space of {\it homogeneous polynomials which are weakly continuous on bounded sets}, denoted by $\mathcal P_w(^nE, F)$. Recall that a polynomial $P\in \mathcal  P(^n E, F)$ is compact if maps the unit ball of $E$ into a relatively compact set in $F$ and that $P$ is in $\mathcal P_w(^nE, F)$ if maps bounded weak convergent nets into convergent nets. For linear operators both properties, to be compact and to be weakly continuous on bounded sets, produce the same subspace. For $n$-homogeneous polynomials with $n>1$, that coincidence is no longer true. Although any polynomial in $\mathcal P_w(^nE, F)$ is compact (as it can be derived from results in \cite{AP} and \cite{AHV}), the reverse inclusion fails. This is due to the fact that  continuous polynomials are not, in general, weak-to-weak continuous.  Then, every scalar-valued continuous polynomial is compact but it is not necessarily weakly continuous on bounded sets, as the standard example $P(x)=\sum_k x_k^2$, for all $x=(x_k)_k\in \ell_2$, shows. With this in mind, our main purpose is to discuss whether $\mathcal P_w(^nE, F)$ is an $M$-ideal in $\mathcal P(^n E, F)$. In \cite{Dim}, the first author studied the analogous question when $F$ is the scalar field. We will see that the vector-valued case is not a mere generalization of the scalar-valued case.

The problem of stating if $\mathcal P_w(^nE, F)$ is a proper subspace of $\mathcal P(^nE, F)$ is nontrivial at all. However, when this is not the situation our question is trivially answered. We refer the reader to \cite{AF,BR-98, GG, GJ},
where the equality $\mathcal P_w(^nE, F)=\mathcal P(^nE, F)$ is studied.
%where the authors studied for which values of $n$, all the continuous $n$-homogeneous polynomials on a fixed Banach space, are weakly continuous on bounded sets.

As it happens for $n$-homogeneous polynomials in the scalar-valued case, the value of $n$ for which  $\mathcal P_w(^nE, F)$ has the chance to be a nontrivial $M$-ideal in  $\mathcal P(^nE, F)$ cannot be chosen arbitrarily. Thus, our firsts efforts are focused to discuss this matter. In order to do so, following \cite{HWW} and \cite{Dim}, we define the essential norm of a vector-valued polynomial $P$ as the distance from $P$ to the space $\mathcal P_w(^nE, F)$. Also we describe the extreme points of the ball of the dual space of $\mathcal P_w(^nE, F)$. Then, combining this with properties of the essential norm we obtain the range within we may expect to find an $M$-structure.  When  $\mathcal P_w(^nE, F)$ is an $M$-ideal in  $\mathcal P(^nE, F)$, the essential norm allows us to obtain a Bishop-Phelps type theorem. We use this result to study the existence of farthest points and densely remotal sets. These concepts are related to geometric properties such us the existence of exposed points, the Mazur intersection property and norm attaining functions, see \cite{As, Ed}. These  results appear in Section 1.

Section 2 is dedicated to give sufficient conditions on $E$ and $F$ so that  $\mathcal P_w(^nE, F)$ is an $M$-ideal in  $\mathcal P(^nE, F)$.  The main requirement stays around the concept of shrinking approximations of the identity. When $F$ is an $M_\infty$-space, without any further assumption on the space $E$, we prove that $\mathcal P_w(^nE, F)$ is a nontrivial $M$-ideal in $\mathcal P(^nE, F)$ for all but one possible value of $n$ in the range of interest. For the remaining value of $n$, we obtain the result when $E$ satisfies some additional conditions, see Propositions~\ref{F Moo-space} and~\ref{F Moo-space-bis}.

In Section 3, we focus our attention on  classical sequence spaces $E$ and $F$, for $E=\ell_p$ ($1\le p<\infty$) and $F=\ell_q$ or $F=d(w,q)$ a Lorentz sequence space, ($1\le q<\infty$). The questions of whether $\mathcal K(\ell_p, \ell_q)$ is an $M$-ideal in $\mathcal L(\ell_p, \ell_q)$ and $\mathcal K(\ell_p, d(w,q))$  is an $M$-ideal in  $\mathcal L(\ell_p, d(w,q))$ were previously addressed in \cite{HWW} and \cite{Oja-91}. In \cite{Dim}, it was studied when $\mathcal P_w(^n\ell_p)$ is an $M$-ideal in  $\mathcal P(^n\ell_p)$. We analyze here when $\mathcal P_w(^n\ell_p, \ell_q)$ is an $M$-ideal in  $\mathcal P(^n\ell_p, \ell_q)$ and when $\mathcal P_w(^n\ell_p, d(w,q))$ is an $M$-ideal in  $\mathcal P(^n\ell_p, d(w,q))$. Giving conditions on $n, p, q$ and $w$ we solve the problem for all the possible situations.

In the last section we study the property $(M)$, introduced by Kalton in \cite{K} for Banach spaces,  developed later for operators by Kalton and  Werner in \cite{KW} and finally generalized to the scalar-valued polynomial setting in \cite{Dim}. Here, we present a natural extension to the vector-valued polynomial setting of the notions mentioned before and establish the connection this property has with our main problem. We apply the results obtained to give examples of $M$-ideals in vector-valued polynomial spaces defined on Bergman spaces.

Before proceeding, we fix some notation and give basic definitions. Every time we write $X, E$  or $F$ we will be considering Banach spaces over the real or complex field, $\mathbb{K}$. The closed unit ball of $X$ will be noted by $B_X$ and the unit sphere by $S_X$. Also, if $x\in X$ and $r>0$, $B(x,r)$ will stand for the closed ball in $X$ with center at $x$ and radius $r$. As usual, $X^*$ and $X^{**}$ will be the notations for the dual and bidual of $X$, respectively. The space of linear bounded operators from $E$ to $E$ will be noted by $\mathcal{L}(E)$ and its subspace of compact mappings will be noted by $\mathcal{K}(E)$.

A function $P\colon E\to F$ is said to be an $n$-homogeneous polynomial if there exists a (unique) symmetric $n$-linear form $\overset\vee P\colon\underbrace{E\times\cdots\times E}_n\to
F$ such that
$$
P(x)=\overset\vee P(x,\dots,x),
$$
for all $x\in E$. For scalar-valued mappings we will write $\mathcal{P}(^nE)$ instead of $\mathcal P(^nE,F)$ to denote the space of all continuous $n$-homogeneous polynomials from $E$ to $\mathbb{K}$.  The space $\mathcal P(^nE,F)$ endowed with the supremum norm
$$
\|P\|=\sup\{\|P(x)\|_F\colon \, x\in B_E\},
$$ is a Banach space. We may write $\|P(x)\|$ instead of $\|P(x)\|_F$ unless we prefer to emphasize the space where the norm is taken.

Every polynomial $P\in\mathcal{P}(^nE,F)$ has two natural mappings associated: the {\it linear adjoint} or {\it transpose} $P^*\in \mathcal{L}(F^*,\mathcal{P}(^nE))$ which
is given by
$$
(P^*(y^*))(x)=y^*(P(x)), \textrm{ for every }x\in E \textrm{ and } y^*\in F^*,
$$
and the polynomial $\overline{P}\in \mathcal P(^n E^{**}, F^{**})$, the canonical extension of $P$ from $E$ to $E^{**}$ obtained by weak-star density, known as the Aron-Berner extension of $P$ \cite{AB}. For each $z\in E^{**}$, $e_z$ will refer to the application given by $e_z(P)=\overline{P}(z)$; for $x\in E$, $e_x$ denotes the evaluation map.

Besides the subspace of  weakly continuous $n$-homogeneous polynomials on bounded sets which was already introduced, we will consider the following classes. The first one is the space of $n$-homogeneous polynomials that are weakly continuous on bounded sets at 0, which consists on those polynomials mapping bounded weakly null nets into null nets. This space will be denoted by  $\mathcal{P}_{w0}(^nE, F)$. We also have the subspace formed by polynomials of finite type, which are of the form   $\sum_{j=1}^N (x^*_j)^n\cdot y_j$, with $x^*_j\in E^*$, $y_j\in F$ for all $j=1,\ldots, N$ and $N\in \mathbb N$. The space of finite type $n$-homogeneous polynomials will be denoted by
$\mathcal{P}_f(^nE, F)$. Its closure (in the supremum norm) is the space of approximable $n$-homogeneous polynomials which will be noted by $\mathcal{P}_A(^nE, F)$. When $F$ is $\mathbb K$ we omit $F$ and write
$\mathcal{P}_{w0}(^nE), \mathcal{P}_f(^nE)$ or $\mathcal{P}_A(^nE)$ for instance.

Recall that if $E$ does not contain a subspace isomorphic to $\ell_1$, then, for any Banach space $F$, $\mathcal{P}_w(^nE, F)$ coincides with the space of weakly sequentially continuous polynomials $\mathcal{P}_{wsc}(^nE, F)$
\cite[Proposition 2.12]{AHV}. The space of $n$-homogeneous polynomials that are weakly sequentially continuous at 0 will be denoted $\mathcal{P}_{wsc0}(^nE, F)$. We refer to \cite{Di, Mu} for the necessary background on polynomials on Banach spaces.
\smallskip

Related to the study of $M$-structures there are two relevant geometric properties that we will use repeatedly. The first one is a well-known characterization, called {\it the 3-ball property}, given by Alfsen and Effros in \cite[Theorem A]{AE} to which the main part of their article is dedicated, see also \cite[Theorem I.2.2 (iv)]{HWW}:
\medskip

\noindent {\bf Theorem A.} {\it Suppose that $J$  is a closed subspace of $X$. The following are equivalent:
\begin{enumerate}
\item[(i)] $J$ is an $M$-ideal.
\item[(ii)] $J$ satisfies the 3-ball property: for every $x_1, x_2, x_3\in X$ and positive numbers $r_1, r_2, r_3$ such that

$$
\bigcap_{j=1}^3 B(x_j, r_j)\ne \emptyset \qquad  and \qquad B(x_j, r_j)\cap J\ne \emptyset,\qquad j=1, 2, 3,
$$
it holds that
$$
\bigcap_{j=1}^3 B(x_j, r_j+\varepsilon )\cap J\ne \emptyset  \qquad  for\ all\ \varepsilon >0.
$$

\item[(iii)] $J$ satisfies the (restricted) 3-ball property: for every $y_1,y_2,y_3\in B_J$, $x\in B_X$ and $\varepsilon >0$, there exists $y\in J$ satisfying
	$$
	\|x+y_j-y\|\leq 1+\varepsilon, \qquad j=1,2,3.
	$$
\end{enumerate}
}
\medskip
Note that one of the benefits of having the 3-ball property is that we have a criterium to decide if a closed subspace of a Banach space $X$ is an $M$-ideal in terms of an intersection of balls in $X$. Thus, there is no need to appeal to the dual space to determine the existence of an $M$-structure. The $2$-ball property is not sufficient to this end, see \cite{HWW}. When a closed subspace of $X$ satisfies the 2-ball property we say that we are in presence of a {\bf semi $M$-ideal} structure.

The second property we referred, provides us with a nice description of the extreme points of the unit ball of $X^*$ in terms of the sets of the extreme points of the unit balls of $J^\perp$ and $J^*$, if $J$ is an $M$-ideal in $X$, see \cite[Lemma 1.5]{HWW}. As usual $Ext(B_{X})$ denotes the set of extreme points of the unit ball of a Banach space $X$.
\medskip

\noindent {\bf Theorem B.} {\it Suppose that $J$  is an $M$-ideal in $X$. Then, the extreme points of the unit ball of $X^*$ satisfy
$$
Ext(B_{X^*})=Ext(B_{J^\perp})\cup Ext(B_{J^*}).
$$
}

Many authors investigated $M$-structures on Banach spaces. Hardmand, Werner and Werner summarized the main results on this topic in their monograph \cite{HWW}. The reader will find out that it is a very clear and well-organized survey on $M$-ideals. Along this paper, we will recourse to the ideas and results in it.
\bigskip

\section{General results}

It is natural to begin our research with  vector-valued polynomial versions of basic results stated for linear operators in \cite[Propositions VI.4.2 and VI.4.3]{HWW} and for scalar-valued polynomials in \cite[Propositions 1.1 and 1.2]{Dim}. We omit the proofs since they are straightforward.

\begin{proposition} \label{M-sumando}
\begin{enumerate}
\item[(a)]If $\mathcal{P}_w(^nE,F)$ is an $M$-summand in $\mathcal{P}(^nE,F)$,
then $\mathcal{P}_w(^nE,F)=\mathcal{P}(^nE,F)$.

\item[(b)] If $\mathcal{P}_w(^nE,F)$ is an $M$-ideal in $\mathcal{P}(^nE,F)$ and $E_1\subset E$, $F_1\subset F$ are 1-complemented subspaces, then
$\mathcal{P}_w(^nE_1,F_1)$ is an $M$-ideal in $\mathcal{P}(^nE_1,F_1)$.

\item[(c)] The class of Banach spaces $E$ and $F$ for which $\mathcal{P}_w(^nE,F)$ is an $M$-ideal in
$\mathcal{P}(^nE,F)$ is closed with respect to the Banach-Mazur
distance.
\end{enumerate}
\end{proposition}

The knowledge of the extreme points of the unit ball of a Banach space provides a crucial tool in the geometric study of the space. We borrow some ideas of \cite{HWW} and \cite{Dim} to examine the extreme points of the unit ball of the dual spaces: $\mathcal{P}(^nE,F)^*$ and $\mathcal{P}_w(^nE,F)^*$.

Note that if $J$ is a subspace of $\mathcal{P}(^nE,F)$ that contains $\mathcal{P}_f(^nE,F)$, then
$e_x\otimes y^*\in J^*$ is a norm one element, for all $x\in S_E$ and $y^*\in S_{F^*}$. Indeed, the  application $e_x\otimes y^*$ belongs to $B_{J^*}$ and since $J$ contains all finite type $n$-homogeneous polynomials, it contains the elements of the form $(x^*)^n\cdot y$, for every  $x^*\in E^*$ and $y\in F$, thus $\|e_x\otimes y^*\|=1$.

\begin{proposition}\label{extremales}
\begin{enumerate}
\item[(a)] If $J$ is a subspace of $\mathcal{P}(^nE,F)$  that contains all finite type $n$-homogeneous polynomials, then
$$
Ext B_{J^*}\subset \overline{\big\{ e_x\otimes y^*
:x\in S_E,\, y^*\in S_{F^*}\big\}}^{w^*},
$$
where $w^*$ designates the topology $\sigma(J^*,J)$.

\item[(b)]
For the particular case $J=\mathcal{P}_w(^nE,F)$ we can be more precise:
$$
Ext B_{\mathcal{P}_w(^nE,F)^*}\subset \{e_z\otimes y^*:z\in S_{E^{**}},\, y^*\in S_{F^*}\big\}.
$$
\end{enumerate}
\end{proposition}

\begin{proof}
(a) Through Hahn-Banach theorem and the comment made above, it easily follows that
$$
B_{J^*}=\overline{\Gamma\big\{ e_x\otimes y^*
:x\in S_E,\, y^*\in S_{F^*}\big\}}^{w^*}.
$$

Now, by Milman's theorem \cite[Theorem 3.41]{FHHMPZ} we derive the desired inclusion:
$$
Ext B_{J^*}\subset \overline{\big\{ e_x\otimes y^*
:x\in S_E,\, y^*\in S_{F^*}\big\}}^{w^*}.
$$
(b) Suppose that $J=\mathcal{P}_w(^nE,F)$. Let us see that $\overline{\{ e_x\otimes y^*:x\in
S_E,\, y^*\in S_{F^*}\}}^{w^*}\subset\{ e_z\otimes y^*:z\in B_{E^{**}},\, y^*\in B_{F^*}
\big\}$. If $\Phi\in\overline{\{ e_x\otimes y^*:x\in
S_E,\, y^*\in S_{F^*}\}}^{w^*}$, then there exist nets $\{x_\alpha\}_\alpha$ in $S_E$  and $\{y^*_\alpha\}_\alpha$ in $S_{F^*}$ such that $e_{x_\alpha}\otimes y_\alpha^*\overset{w^*}{\rightarrow}\Phi$.
Without loss of generality, we may assume that $\{x_\alpha\}_\alpha$ is
$\sigma(E^{**}, E^*)$-convergent to an element $z$ in $B_{E^{**}}$ and $\{y^*_\alpha\}_\alpha$ is $\sigma(F^{*}, F)$-convergent to an element $y^*$ in $B_{F^{*}}$.

Note that for any $P\in \mathcal{P}_w(^nE,F)$, its Aron-Berner extension $\overline P$ belongs to $\mathcal{P}(^nE^{**},F)$ (see for instance \cite[Proposition 2.5]{CaLa}) and the compacity of $P$ implies that $\overline P$ is $w^*$-continuous on bounded sets. Then, we have that $y_\alpha^*\left(\overline{P}(x_\alpha)\right)\to y^*\left(\overline{P}(z)\right)$, for every $P\in \mathcal{P}_w(^nE,F)$. Thus, $e_{ x_\alpha}\otimes y_\alpha^*\overset{w^*}{\rightarrow}e_z\otimes y^*$ and therefore, $\Phi=e_z\otimes y^*$. When $\Phi$ is a norm one element we have that both $z$ and $y$ are elements in the respective unit spheres $S_{E^{**}}$ and $S_{F^*}$. Now, the result follows.
\end{proof}

%\vskip .3cm

In \cite{Dim}, the notion of the essential norm was extended from operators to scalar-valued polynomials and was used to determine that $\mathcal P_w(^n E)$ may be a nontrivial $M$-ideal of $\mathcal P(^n E)$ for at most only one value of $n$. For vector-valued polynomials, also through the essential norm, we obtain a finite range of possible values of $n$ for which $\mathcal P_w(^n E, F)$  has the chance to be a nontrivial $M$-ideal of $\mathcal P(^n E, F)$. Recall that the essential norm of a linear operator $T$ is the distance from $T$ to the subspace of compact operators. When $\mathcal K(E,F)$ is an $M$-ideal in $\mathcal L(E,F)$, there is an explicit alternative formula to compute this essential norm \cite[Proposition VI.4.7]{HWW}. Now  we proceed to discuss de degrees of homogeneity for which our problem might have a nontrivial solution.

\begin{definition}
Let $P$ be an $n$-homogeneous polynomial $P\in\mathcal{P}(^nE,F)$. The {\bf
essential norm} of $P$ is defined by
$$
\|P\|_{es}=d(P,\mathcal{P}_w(^nE,F))=\inf\{\|P-Q\|: Q\in
\mathcal{P}_w(^nE,F)\}.
$$
\end{definition}

In order to obtain a good description of the essential norm, we will make use of the transpose of a polynomial. Note that if $P\in \mathcal P(^nE,F)$ and we denote by $L_P:\bigotimes_{\pi_s}^{n,s} E\to F$ the linearization of $P$, where $\pi_s$ is the projective symmetric tensor norm; then $P^*$ is the usual adjoint of $L_P$.

\begin{lemma}\label{w^*}
If $P\in\mathcal{P}_w(^nE,F)$ then $P^*$ belongs to $\mathcal{L}(F^*,\mathcal{P}_w(^nE))$ and it is $w^*$-continuous on bounded sets.
\end{lemma}

\begin{proof}
If $P\in\mathcal{P}_w(^nE,F)$ then $P$ is compact and $P^*\in \mathcal{L}(F^*,\mathcal{P}_w(^nE))$. By \cite[Proposition 3.2]{ASch}, $P^*$ is a compact operator. Since $P^*=L_P^*$ it follows that $L_P$ is compact and its adjoint $P^*$ is $w^*$-continuous.
\end{proof}

% As it was shown in \cite{Dim}, for homogeneous polynomials there is

Now we can obtain an alternative formula for the essential norm in the case that there is an $M$-structure.

\begin{proposition}\label{norma esencial}
Suppose $\mathcal{P}_w(^nE,F)$ is an $M$-ideal in $\mathcal{P}(^nE,F)$. Then,
for any $P\in\mathcal{P}(^nE,F)$,

$$\|P\|_{es}=\max\{w(P), w^*(P)\},$$
where
\begin{eqnarray*}
w(P)&=&\sup\big\{\limsup\|P(x_\alpha)\|: \|x_\alpha\|=1,\,
x_\alpha\overset{w}{\rightarrow} 0\big\}\quad and \\
w^*(P)&=&\sup\big\{\limsup\|P^*(y_\alpha^*)\|: \|y_\alpha^*\|=1,\,
y^*_\alpha\overset{w^*}{\rightarrow} 0\big\}.
\end{eqnarray*}
\end{proposition}

\begin{proof}
Let $P\in\mathcal{P}(^nE,F)$. For any $Q\in\mathcal{P}_w(^nE,F)$ and
%for any normalized weakly null net $\{x_\alpha\}_\alpha$, we have
%$$
%\|P-Q\|\geq\big\|(P-Q)(x_\alpha)\big\|\geq \big\|P(x_\alpha)\big\|-
%\big\|Q(x_\alpha)\big\|.
%$$
%Since $\big\|Q(x_\alpha)\big\|\to 0$ it follows that
%$\|P-Q\|\geq\limsup\|P(x_\alpha)\|$ and thus $\|P\|_{es}\geq w(P)$.
%
%Also,
for any normalized weak-star null net $\{y_\alpha^*\}_\alpha$, it holds
$$
\|P-Q\|=\|P^*-Q^*\|\geq\big\|(P^*-Q^*)(y^*_\alpha)\big\|\geq \big\|P^*(y^*_\alpha)\big\|-
\big\|Q^*(y^*_\alpha)\big\|.
$$
Since, by Lemma~\ref{w^*}, $\big\|Q^*(y^*_\alpha)\big\|\to 0$ it follows that
$\|P-Q\|\geq\limsup\|P^*(y^*_\alpha)\|$ and thus $\|P\|_{es}\geq w^*(P)$.

The other inequality follows analogously. Thus,
$$\|P\|_{es}\ge\max\{w(P), w^*(P)\}.$$

Now suppose that $\mathcal{P}_w(^nE,F)$ is an $M$-ideal in
$\mathcal{P}(^nE,F)$. Then we have
$$
ExtB_{\mathcal{P}(^nE,F)^*}=ExtB_{\mathcal{P}_w(^nE,F)^\perp}\cup
ExtB_{\mathcal{P}_w(^nE,F)^*}.
$$

The essential norm of $P$, $\|P\|_{es}$, is the norm of the class of
$P$ in the quotient space $\mathcal{P}(^nE,F)/\mathcal{P}_w(^nE,F)$ and
the dual of this quotient can be isometrically identified with
$\mathcal{P}_w(^nE,F)^\perp$. Then, there exists $\Phi\in Ext
B_{\mathcal{P}_w(^nE,F)^\perp}$ such that $\Phi(P)=\|P\|_{es}$. So,
$\Phi\in Ext B_{\mathcal{P}(^nE,F)^*}$ and, by Proposition \ref{extremales} (a),
$\Phi\in \overline{\big\{ e_x\otimes y^*:x\in S_E, y^*\in S_{F^*}\big\}}^{w^*}$.

Chose nets $\{x_\alpha\}_\alpha$ in $S_E$ and $\{y_\alpha^*\}_\alpha$ in $S_{F^*}$ such that
$e_{x_\alpha}\otimes y^*_\alpha\overset{w^*}{\rightarrow}\Phi$, where $w^*$ means the
topology $\sigma(\mathcal{P}(^nE,F)^*,\mathcal{P}(^nE,F))$. In passing
to appropriate subnets, we can suppose that $\{x_\alpha\}_\alpha$ is
$\sigma(E^{**}, E^*)$-convergent to an element $z$ in $B_{E^{**}}$ and $\{y_\alpha^*\}_\alpha$ is $\sigma(F^{*}, F)$-convergent to an element $y$ in $B_{F^{*}}$.

For any $x^*\in E^*$ and $y\in F$, the polynomial
$(x^*)^n\cdot y$ belongs to $\mathcal{P}_w(^nE,F)$. This gives
$$ 0=\Phi((x^*)^n\cdot y)=\lim_\alpha x^*(x_\alpha)^ny_\alpha^*(y)=
z(x^*)^n y^*(y).$$

So it should be $z=0$ or $y^*=0$. In the first case, $\{x_\alpha\}_\alpha$ is weakly null and
$$
\|P\|_{es}=\Phi(P)=\lim_\alpha y_\alpha^*\left(P(x_\alpha)\right)\le \limsup\left\|P(x_\alpha)\right\|\leq w(P).
$$

In the second case, $\{y_\alpha^*\}_\alpha$ is weak-star null and it follows similarly that $\|P\|_{es}\leq w^*(P)$.
\end{proof}

As in the scalar-valued polynomial case this result enable us to narrow the possible values of $n$ for which  $\mathcal{P}_w(^nE,F)$ could be an $M$-ideal in $\mathcal{P}(^nE,F)$. To see this, first note that the arguments from \cite{BR-98} and \cite{AD} used in the comments before Remark 1.8 of \cite{Dim} also work for vector-valued polynomials. Thus, we obtain:

\begin{remark}\label{n-unico}\rm
For Banach spaces $E$ and $F$, either $\mathcal{P}_w(^kE,F)=\mathcal{P}_{w0}(^kE,F)=\mathcal{P}(^kE,F)$,
for all $k$, or there exists $n\in\mathbb{N}$ such that:
%$\mathcal{P}_w(^kE)=\mathcal{P}_{w0}(^kE)=\mathcal{P}(^kE)$, for
%every $k\in\mathbb{N}$, or there exists $n\in\mathbb{N}$ such that
%the following holds:
\begin{itemize}
\item $\mathcal{P}_w(^kE,F)=\mathcal{P}_{w0}(^kE,F)=\mathcal{P}(^kE,F)$,
for all $k< n$.
\item $\mathcal{P}_w(^nE,F)=\mathcal{P}_{w0}(^nE,F)\subsetneqq
\mathcal{P}(^nE,F)$.
\item $\mathcal{P}_w(^kE,F)\subsetneqq \mathcal{P}_{w0}(^kE,F)\subset \mathcal{P}(^kE,F)$,
for all $k> n$.
\end{itemize}

When this value of $n$ does exist, we call it \textbf{the critical degree of $(E,F)$} and denote $n=cd(E,F)$. For the case $F=\mathbb{K}$ we write $cd(E)$ instead of $cd(E,\mathbb{K})$.
\end{remark}

Therefore, if there exists a polynomial from $E$ to $F$ which is not weakly
continuous on bounded sets, the critical degree is the minimum of
all $k$ such that $\mathcal{P}_{w}(^kE,F)\not=\mathcal{P}(^kE,F)$.

Observe that if  a scalar-valued polynomial $P\in\mathcal{P}(^nE)$ is not weakly continuous on bounded sets then, for any $y\in F$, $y\not=0$, the polynomial $x\mapsto P(x)y$ belongs to  $\mathcal{P}(^nE,F)$ and it is not weakly continuous on bounded sets. This says that, for any Banach space $F$,
$$
cd(E,F)\leq cd(E).
$$

Note also that $cd(E,F)$ could be much smaller than $cd(E)$. For instance, $cd (\ell_p,c_0)=1$ while $cd(\ell_p)$ is the integer number satisfying $p\leq cd(\ell_p)< p+1$.

\begin{lemma}\label{compacto} Let  $P\in\mathcal{P}(^nE,F)$ be a compact polynomial.
\begin{enumerate}
\item[(a)] If $n<cd(E)$ then $P$ is weakly continuous on bounded sets.

\item[(b)] $w^*(P)=0$.
\end{enumerate}
\end{lemma}

\begin{proof}\textcolor{red}{}
(a) If $n<cd(E)$, then every scalar-valued $n$-homogeneous polynomial on $E$ is weakly continuous on bounded sets. Then, $P$ is weak-to-weak continuous on bounded sets. So, for any bounded net $\{x_\alpha\}_\alpha$ in $E$ such that $x_\alpha\overset{w}{\rightarrow} x$, we have $P(x_\alpha)\overset{w}{\rightarrow} P(x)$. Being $P$ compact, the bounded net $\{P(x_\alpha)\}_\alpha$ should have a convergent subnet. By a canonical argument we derive that $P(x_\alpha){\rightarrow} P(x)$ and thus $P$ is weakly continuous on bounded sets.

(b) This is a consequence of the proof of Lemma \ref{w^*}.
\end{proof}

\begin{proposition}\label{compacto y w0}
Every polynomial in $\mathcal{P}(^nE,F)$ which is weakly continuous on bounded sets at 0 and compact is weakly continuous on bounded sets if and only if $n\leq cd(E)$.
\end{proposition}

\begin{proof}
If $n> cd(E)$, there exists a polynomial $p\in \mathcal{P}_{w0}(^nE)\setminus \mathcal{P}_{w}(^nE)$. Then, for a fixed $y\in F$ the polynomial $P(x)=p(x) y$ is weakly continuous on bounded sets at 0 and compact but it is not weakly continuous on bounded sets.

Reciprocally, let $n\leq cd(E)$ and let $P\in\mathcal{P}(^nE,F)$ be a polynomial weakly continuous on bounded sets at 0 and compact. We know from \cite[Proposition 3.4]{ASch} that, for $0< k<n$, any derivative $d^kP(x)$ is compact. Thus, by Lemma \ref{compacto} (a), we obtain that $d^kP(x)$ is weakly continuous on bounded sets, for all  $0< k<n$. This fact together with the hypothesis of $P$ being weakly continuous on bounded sets at 0 implies that $P$ is weakly continuous on bounded sets.
\end{proof}

By Lemma \ref{compacto} (b), if $P\in\mathcal{P}(^nE,F)$ is weakly continuous on bounded sets at 0 and compact then $w(P)=w^*(P)=0$. If,  in addition, $\mathcal{P}_w(^nE,F)$ is an $M$-ideal in $\mathcal{P}(^nE,F)$, Proposition \ref{norma esencial} states that $\|P\|_{es}=0$ and so $P$ is weakly continuous on bounded sets. Thus, by
Proposition \ref{compacto y w0}, it should be $n\leq cd(E)$. Therefore, we have:

\begin{corollary} \label{grado de E}
If $\mathcal{P}_w(^nE,F)$ is an $M$-ideal in $\mathcal{P}(^nE,F)$, then $n\leq cd(E)$.
\end{corollary}

Clearly, if $n< cd(E,F)$, $\mathcal{P}_w(^nE,F)$ is a trivial $M$-ideal in $\mathcal{P}(^nE,F)$. So the problem proposed is worth being studied for polynomials of degree $n$, with $cd(E,F)\leq n \leq cd(E)$.
\medskip

The fact that $\mathcal{P}_w(^nE,F)$ is an $M$-ideal in $\mathcal{P}(^nE,F)$ has some incidence in the set of polynomials whose Aron-Berner extension attains the norm. As we have for scalar-valued polynomials \cite[Proposition 1.10]{Dim}, the following version of \cite[Proposition VI.4.8]{HWW} is a Bishop-Phelps type result for vector-valued polynomials. The proof is omitted since it can be obtained as a slight modification of the proof given in \cite{Dim}.

\begin{proposition}\label{nunca denso}
Let $E$ and $F$ be Banach spaces and suppose that $\mathcal{P}_w(^nE,F)$ is an $M$-ideal in $\mathcal{P}(^nE,F)$.
\begin{enumerate}
\item[(a)] If $P\in \mathcal{P}(^nE,F)$ is such that its Aron-Berner extension $\overline{P}$
does not attain its norm at $B_{E^{**}}$, then $\|P\|=\|P\|_{es}$.
\item[(b)] The set of polynomials in $\mathcal{P}(^nE,F)$ whose
Aron-Berner extension does not attain the norm is nowhere dense in
$\mathcal{P}(^nE,F)$.
\end{enumerate}
\end{proposition}

We finish this section relating norm attaining polynomials with farthest points and remotal sets. The study of the existence of farthest points in a set of a Banach space can be traced to the articles of Asplund \cite{As} and Edelstein \cite{Ed}. This concept is related to several geometric properties of the space, like the existence of exposed points and the Mazur intersection property.

Perhaps some definitions are in order. Let $J$ be a subspace of a Banach space $X$. Fix $x\in X$, the {\bf farthest distance} from $x$ to the unit ball of $J$ is given by
$$
\rho (x,B_{J})=\sup\{\|x-y\|:\ y\in B_{J}\}.
$$
A point $x\in X$ has a {\bf farthest point} in $B_{J}$ if there exists $y\in B_{J}$ such that $\|x-y\| = \rho(x,B_{J})$. The set of points in $X$ having farthest points in $B_{J}$ is denoted by $R(B_{J})$. Then we have:
$$
R(B_{J})=\left\{x\in X:\ \exists\, y\in B_{J} \textrm{ such that } \|x-y\| = \rho(x,B_{J})\right\}.
$$
It is said that $B_{J}$ is {\bf densely remotal} in $X$ if $R(B_{J})$ is dense in $X$ and it is {\bf almost remotal} in $X$ if $R(B_{J})$ contains a dense $G_{\delta}$ set.

In \cite{BLR}, Bandyopadhyay, Lin and Rao studied dense remotality of the ball of $\mathcal{K}(E,F)$ in the space $\mathcal{L}(E,F)$. Adapting some of their ideas and applying the previous proposition, in Corollary~\ref{almost remotal}, we obtain a result about almost remotality of $B_{\mathcal{P}_w(^nE,F)}$ in $\mathcal{P}(^nE,F)$.

First, observe that for any $P\in\mathcal{P}(^nE,F)$ we have that
$$
\rho\left(P,B_{\mathcal{P}_w(^nE,F)}\right)=\|P\| + 1.
$$

Indeed, it is clear that $\rho\left(P,B_{\mathcal{P}_w(^nE,F)}\right)\le\|P\| + 1$, for every $P\in\mathcal{P}(^nE,F)$ and the equality is obvious for the polynomial $P\equiv 0$.
For the reverse inequality, given $P\in\mathcal{P}(^nE,F)$, $P\not\equiv 0$, and $\varepsilon >0$, fix $x\in S_E$ and $y^*\in S_{F^*}$ such that $y^*(P(x))>(1-\varepsilon ) \|P\|$. Now, take $y\in S_F$ and $x^*\in S_{E^*}$ satisfying $y^*(y)>1-\varepsilon$ and $x^*(x)=1$ and
consider the polynomial
 $Q= - (x^*)^n\cdot y\in B_{\mathcal{P}_w(^nE,F)}$.

Then, we have
\begin{eqnarray*}
\|P-Q\| & = & \|P+(x^*)^n\cdot y\| \ge \left|y^*(P(x)) + x^*(x)^n  y^*(y) \right|\\
&=& y^*(P(x)) + y^*(y) > (1-\varepsilon) \left(\|P\| + 1\right),
\end{eqnarray*}
for all $\varepsilon >0$, which proves the claim.
\medskip

The relation between norm attaining linear functions and the sets of operators which admit farthest points in the unit ball of the space of compact operators was studied in \cite{BLR}. To simplify our statements
let us introduce the following notations:
\begin{eqnarray*}
NA\left(\mathcal{P}(^nE,F)\right)&=&\{ P\in \mathcal{P}(^nE,F):\ P \textrm{ attains its norm at }B_E\},\\
AB-NA\left(\mathcal{P}(^nE,F)\right)&=&\{ P\in \mathcal{P}(^nE,F):\ \overline{P} \textrm{ attains its norm at }B_{E^{**}}\}.
\end{eqnarray*}

\begin{proposition}\label{NA-remotal}
 $NA\left(\mathcal{P}(^nE,F)\right)\subset R\left(B_{\mathcal{P}_w(^nE,F)}\right)$.
\end{proposition}

\begin{proof}
By the previous observation, it is plain that the polynomial $P\equiv 0$ belongs to $R\left(B_{\mathcal{P}_w(^nE,F)}\right)$. Now, if $P\in NA\left(\mathcal{P}(^nE,F)\right)$, $P\not\equiv 0$, there exists $x\in S_E$ such that $\|P(x)\|=\|P\|$. Let $x^*\in S_{E^*}$ satisfying $x^*(x)=1$.

Consider the polynomial $Q= - (x^*)^n\cdot \frac{P(x)}{\|P\|}\in B_{\mathcal{P}_w(^nE,F)}$. So $Q$ is a farthest point for $P$ because
$$
\|P-Q\|  =  \left\|P+(x^*)^n\cdot \frac{P(x)}{\|P\|}\right\| \ge \left\|P(x) +  \frac{P(x)}{\|P\|} \right\|=\|P\| + 1.
$$
\end{proof}

In \cite{CK}, Choi and Kim proved that if $E$ has the Radon-Nykod\'{y}m property, then the set of norm attaining polynomials of $\mathcal{P}(^nE,F)$ is dense in $\mathcal{P}(^nE,F)$. As a consequence of this result we obtain:

\begin{corollary}
If $E$ has the Radon-Nykod\'{y}m property, then $B_{\mathcal{P}_w(^nE,F)}$ is densely remotal in $\mathcal{P}(^nE,F)$.
\end{corollary}

When $\mathcal{P}_w(^nE,F)$ is an $M$-ideal in $\mathcal{P}(^nE,F)$, the set $R\left(B_{\mathcal{P}_w(^nE,F)}\right)$ does not only contain the set of norm attaining polynomials but it is also contained in the set of all the polynomials whose Aron-Berner extension is norm attaining.

\begin{proposition}\label{AB-remotal}
If $\mathcal{P}_w(^nE,F)$ is an $M$-ideal in $\mathcal{P}(^nE,F)$, then $R\left(B_{\mathcal{P}_w(^nE,F)}\right)\subset AB-NA\left(\mathcal{P}(^nE,F)\right)$.
\end{proposition}

\begin{proof}
Let $P\in R\left(B_{\mathcal{P}_w(^nE,F)}\right)$. So, there exists $Q\in B_{\mathcal{P}_w(^nE,F)}$ such that $\|P-Q\|=\|P\| +1$.  Take $\Phi\in Ext B_{\mathcal{P}(^nE,F)^*}$ satisfying
$$
\Phi(P-Q) =\|P-Q\|=\|P\| +1.
$$
Being $\mathcal{P}_w(^nE,F)$ an $M$-ideal in $\mathcal{P}(^nE,F)$, we should have that
$$
\Phi\in Ext B_{\mathcal{P}_w(^nE,F)^*}\qquad \textrm{or}\qquad\Phi\in Ext B_{\mathcal{P}_w(^nE,F)^\perp}.
$$
If $\Phi\in Ext B_{\mathcal{P}_w(^nE,F)^\perp}$, we  obtain  that $\Phi(P-Q)=\Phi(P)$ and so $\Phi(P)=\|P\|+1$, which is not possible. Hence,
it should be $\Phi\in Ext B_{\mathcal{P}_w(^nE,F)^*}$ and, by Proposition \ref{extremales} (b), $\Phi= e_z\otimes y^*$, for certain $z\in S_{E^{**}}$ and $y^*\in S_{F^*}$. Therefore,
$$
\|P\|+1 =\Phi(P-Q)=y^*(\overline{P}(z))-y^*(\overline{Q}(z))\le \|\overline{P}\| + \|\overline{Q}\|=\|P\|+1.
$$

It follows that $y^*(\overline{P}(z))=\|\overline{P}\|$  and so  $\|\overline{P}(z)\|=\|\overline{P}\|$, meaning that $P\in AB-NA\left(\mathcal{P}(^nE,F)\right)$.
\end{proof}

As a consequence of Propositions \ref{nunca denso}, \ref{NA-remotal} and \ref{AB-remotal}, we obtain:
\begin{corollary}\label{almost remotal}
If $E$ is reflexive and $\mathcal{P}_w(^nE,F)$ is an $M$-ideal in $\mathcal{P}(^nE,F)$, then
$$
R\left(B_{\mathcal{P}_w(^nE,F)}\right)=NA\left(\mathcal{P}(^nE,F)\right),
$$
and thus, $\mathcal{P}(^nE,F) \setminus R\left(B_{\mathcal{P}_w(^nE,F)}\right)$ is nowhere dense. This implies that $B_{\mathcal{P}_w(^nE,F)}$ is almost remotal in $\mathcal{P}(^nE,F)$.
\end{corollary}

\section{Sufficient conditions}

In this section we present several sets of sufficient conditions which enable us to ensure that $\mathcal{P}_w(^nE,F)$ is an $M$-ideal in $\mathcal{P}(^nE,F)$. All of them involve bounded nets of compact operators on $E$. The following lemma and proposition are the vector-valued versions of \cite[Lemma 2.1 and Proposition 2.2]{Dim},  the proofs  of which are analogous to those in \cite{Dim}.

\begin{lemma}\label{Pw}
Let $E$ and $F$ be Banach spaces and suppose that there exists a bounded
net $\{S_\alpha\}_\alpha$ of linear operators from $E$ to $E$
satisfying $S_\alpha^* (x^*) \to x^*$, for all $x^*\in E^*$.
Then, for all $P\in\mathcal{P}_w(^nE,F)$, we have that $\|P-P\circ
S_\alpha\|\to 0$.
\end{lemma}

\begin{proposition}\label{condition}
Let $E$ and $F$ be Banach spaces and let $n=cd(E,F)$. Suppose that there exists a bounded
net $\{K_\alpha\}_\alpha$ of compact operators from $E$ to $E$
satisfying the following two conditions:
\begin{itemize}
\item $K_\alpha^* (x^*) \to x^*$, for all $x^*\in E^*$.
\item For all $\varepsilon >0$ and all $\alpha_0$ there exists $\alpha >\alpha_0$ such that for every $x\in E$,
$$
\|K_\alpha (x)\|^n + \|x-K_\alpha (x)\|^n \leq (1+\varepsilon)
\|x\|^n.
$$
\end{itemize}
Then,
$\mathcal{P}_w(^nE,F)$ is an $M$-ideal in $\mathcal{P}(^nE,F)$.
\end{proposition}

\begin{remark}\rm
A Banach space $E$ is an $(M_p)$-space ($1\le p\le \infty$) if $\mathcal{K}(E\oplus_p E)$ is an $M$-ideal
in $\mathcal{L}(E\oplus_p E)$. This concept was introduced by Oja and Werner in \cite{OW}. By \cite[Theorem VI.5.3]{HWW}, if $E$  is an $(M_p)$-space with  $p\leq n$, then there exists a bounded net $\{K_\alpha\}_\alpha$ of compact operators from $E$ to $E$ satisfying both conditions of Proposition~\ref{condition}.
\end{remark}

Recall that a Banach space $E$ has a finite dimensional decomposition $\{E_j\}_j$ if each $E_j$ is a finite dimensional subspace of $E$ and  every $x\in E$ has a unique representation of the form
$$
x=\sum_{j=1}^\infty x_j,\qquad \textrm{with } x_j\in E_j, \textrm{ for every } j.
$$
Associated to the decomposition there is a bounded sequence of projections $\{\pi_m\}_m$, given by
$\pi_m\left(\sum_{j=1}^\infty x_j\right)=\sum_{j=1}^m x_j.$ The decomposition is called shrinking if $ \pi_m^*(x^*)\to x^*$, for all $x^*\in E^*$.

It is clear that in this case $\{\pi_m\}_m$ is a bounded sequence of compact operators that satisfies the first item of the previous proposition. Thus, for spaces with shrinking finite dimensional decompositions we state
the following simpler version of Proposition~\ref{condition}.

\begin{corollary}\label{fdd}
Let $E$ and $F$ be Banach spaces and let $n=cd(E,F)$. Suppose that $E$ has a shrinking finite dimensional
decomposition with associate projections $\{\pi_m\}_m$ such that:
\begin{itemize}
\item For all $\varepsilon >0$ and all $m_0\in\mathbb{N}$ there
exists $m>m_0$ such that for every $x\in E$,
$$
\|\pi_m (x)\|^n + \|x-\pi_m (x)\|^n \leq (1+\varepsilon) \|x\|^n.
$$
\end{itemize}
Then,
$\mathcal{P}_w(^nE,F)$ is an $M$-ideal in $\mathcal{P}(^nE,F)$.
\end{corollary}

\begin{example} \label{elepe}\rm
Let $E=\bigoplus_{\ell_p}X_m$ and $F=\bigoplus_{\ell_q}Y_m$, where  $X_m$ and $Y_m$ are finite
dimensional spaces and $1<p,q<\infty$. From \cite{GJ} we can derive that the critical degree is the integer number $cd(E,F)$ satisfying
$\frac{p}{q}\leq cd(E,F)<\frac{p}{q}+1$. The hypothesis of
Corollary \ref{fdd} are fulfilled if $cd(E,F)\geq p$. When this is the case we thus obtain that $\mathcal{P}_w(^nE,F)$ is an $M$-ideal
in $\mathcal{P}(^nE,F)$.
\end{example}

The conditions in the following theorem were inspired by those of \cite[Lemma VI.6.7]{HWW}. They concern bounded nets of compact operators both in $E$ and in $F$. As an example of this result (see Example \ref{elepe2} below) we can consider the values of $n=cd(\ell_p,\ell_q)$ uncovered by Example \ref{elepe}.

\begin{theorem}\label{largo}
Let $E$ and $F$ be Banach spaces and suppose that there exist bounded nets of compact operators $\{K_\alpha\}_\alpha\subset \mathcal{K}(E)$ and $\{L_\beta\}_\beta\subset \mathcal{K}(F)$ and numbers $1<p,q<\infty$ such that:
\begin{itemize}
\item $K_\alpha^*(x^*)\to x^*$, for all $x^*\in E^*$ and $L_\beta (y)\to y$, for all $y\in F$.
\item For all $\varepsilon >0$ and all $\alpha_0$ there exists $\alpha >\alpha_0$ such that for every $x\in E$,
$$
\|K_\alpha (x)\|^p + \|x-K_\alpha (x)\|^p \leq (1+\varepsilon)^p
\|x\|^p.
$$
\item For all $\varepsilon >0$ and all $\beta_0$ there exists $\beta >\beta_0$ such that for every $y_1, y_2\in F$,
$$
\|L_\beta (y_1)+(Id-L_\beta) (y_2)\|^q \leq (1+\varepsilon)^q
\Big(\|y_1\|^q+\|y_2\|^q\Big).
$$
\end{itemize}
Suppose also that $n=cd(E,F)$ satisfies that $p\leq nq$  and $n<cd(E)$. Then,
$\mathcal{P}_w(^nE,F)$ is an $M$-ideal in $\mathcal{P}(^nE,F)$.
\end{theorem}

\begin{proof}
We prove that the 3-ball property holds. Let $P_1, P_2, P_3\in B_{\mathcal{P}_w(^nE,F)}$, $Q\in
B_{\mathcal{P}(^nE,F)}$ and $\varepsilon >0$. Define $P=Q-(Id - L_\beta) Q (Id-K_\alpha)$. We want to show that $P$ is weakly continuous on bounded sets and $\|Q+ P_j-P\|\le 1+\varepsilon$, for $j=1,2,3$, for some convenient choice of $\alpha$ and $\beta$.

To see that $P$ is weakly continuous on bounded sets, we write $P=Q-Q(Id-K_\alpha)+ L_\beta Q (Id-K_\alpha).$ The proof of \cite[Proposition 2.2]{Dim} shows that $Q-Q(Id-K_\alpha)$ is weakly continuous on bounded sets at 0 and since $n=cd(E,F)$, we have that $Q-Q(Id-K_\alpha)$ belongs to $\mathcal{P}_w(^nE,F)$. Also, as $L_\beta Q (Id-K_\alpha)$ is a compact polynomial and  $n<cd(E)$,  Lemma~\ref{compacto} (a) says that it is in $\mathcal{P}_w(^nE,F)$.

Now, to show the $(1+\varepsilon)$-bound, consider  the inequality
$$\|Q+ P_j-P\|\le\|Q + L_\beta P_j K_\alpha -P\| + \|P_j - L_\beta P_j K_\alpha\|.$$ On the one hand, we have:
\begin{eqnarray*}
\|P_j - L_\beta  P_j K_\alpha\|
&\le & \|P_j - P_j K_\alpha\| + \|P_j K_\alpha - L_\beta  P_j K_\alpha\|\\
&\le & \|P_j - P_j K_\alpha\| + \|P_j- L_\beta  P_j\|\|K_\alpha\|^n.
\end{eqnarray*}
By Lemma~\ref{Pw}, $\|P_j - P_j K_\alpha\|\to 0$ with $\alpha$. Also, since $L_\beta$ approximates the identity on compact sets and the $P_j$'s are compact polynomials, we have that $\|P_j- L_\beta  P_j\|\|K_\alpha\|^n \to 0$ with $\beta$, for all $\alpha$.

Furthermore, we can find $\alpha$ and $\beta$ such that:

\begin{eqnarray*}
 \|Q + L_\beta P_j K_\alpha -P\|
& = &\sup_{x\in B_E} \|(Id - L_\beta) Q (Id-K_\alpha)(x) + L_\beta P_j K_\alpha (x)\|\\
& \leq & \sup_{x\in B_E} (1+\varepsilon) \left(\|Q (Id-K_\alpha)(x)\|^q + \|P_j K_\alpha (x)\|^q\right)^{\frac 1q}\\
& \leq & (1+\varepsilon) \sup_{x\in B_E} \left(\|(Id-K_\alpha)(x)\|^{nq} + \| K_\alpha (x)\|^{nq}\right)^{\frac 1q}\\
& \leq & (1+\varepsilon) \sup_{x\in B_E} \left(\|(Id-K_\alpha)(x)\|^{p} + \| K_\alpha (x)\|^{p}\right)^{\frac np}\\
& \leq & (1+\varepsilon) (1+\varepsilon)^{n}= (1+\varepsilon)^{n+1},
\end{eqnarray*}
and the result follows.
\end{proof}

\begin{remark}\rm
If $E$ is an $(M_p)$-space and $F$ is an $(M_q)$-space the conditions about the
nets of compact operators of the previous theorem are fulfilled.
\end{remark}

\begin{example} \label{elepe2}\rm
Let $E=\bigoplus_{\ell_p}X_m$ and $F=\bigoplus_{\ell_q}Y_m$, where  $X_m$ and $Y_m$ are finite dimensional spaces and $1<p,q<\infty$. As we note in Example \ref{elepe}, $cd(E,F)$ is the integer such that $\frac{p}{q}\leq cd(E,F)<\frac{p}{q}+1$. Also we know that $cd(E)$ is the integer satisfying $p\leq cd(E)<p+1$. Now we obtain a result for the cases uncovered by Example \ref{elepe}, since if $n=cd(E,F)<p$ all the hypothesis of the above theorem hold. So,  $\mathcal{P}_w(^nE,F)$ is an $M$-ideal in $\mathcal{P}(^nE,F)$.

In particular, we derive from this example and Example \ref{elepe}, that for every $1<p,q<\infty$, if $n=cd(\ell_p, \ell_q)$, then $\mathcal{P}_w(^n\ell_p,\ell_q)$ is an $M$-ideal in $\mathcal{P}(^n\ell_p,\ell_q)$.
\end{example}

In all the previous results (Proposition \ref{condition}, Corollary \ref{fdd} and Theorem \ref{largo}) the $M$-structure is obtained only in the case $n=cd(E,F)$. On the other hand, by the comments after Corollary \ref{grado de E}, for $\mathcal{P}_w(^nE,F)$ to be a nontrivial $M$-ideal in $\mathcal{P}(^nE,F)$, it is necessary that $cd(E,F)\leq n\leq cd(E)$. Let us show now  some positive results for values of $n$ greater than $cd(E,F)$.

\begin{proposition}\label{F Moo-space}
Let $E$ be a Banach space and $F$ be an $(M_\infty)$-space. If $n<cd(E)$, then $\mathcal{P}_w(^nE,F)$ is an $M$-ideal in $\mathcal{P}(^nE,F)$.
\end{proposition}

\begin{proof}
Being $F$ an $(M_\infty)$-space, by \cite[Theorem VI.5.3]{HWW}, there exists a net
 $\{L_\beta\}_\beta$ contained in the unit ball of $\mathcal{K}(F)$ satisfying $L_\beta (y)\to y$ for all $y\in F$ such that for any $\varepsilon >0$, there exists $\beta_0$ with
\begin{equation}\label{M-infinito}
\|L_\beta (y_1)+(Id-L_\beta) (y_2)\| \leq \left(1+\frac \varepsilon 2\right)
\max\{\|y_1\|, \|y_2\|\},
\end{equation}
for all $\beta \ge \beta_0$ and for any $y_1, y_2\in F$. Let  $P_1, P_2, P_3\in B_{\mathcal{P}_w(^nE,F)}$ and $Q\in B_{\mathcal{P}(^nE,F)}$, we show that with $P=L_\beta Q$, choosing $\beta$ properly, the 3-ball property is satisfied.

First, note that by Lemma~\ref{compacto} (a), $P$ is weakly continuous on bounded sets. Also,
$\|Q + P_j-P\|\le \|Q+ L_\beta P_j -P\| + \|P_j - L_\beta P_j\|$. Reasoning as in Theorem~\ref{largo}, we have that $ \|P_j - L_\beta P_j\|< \frac \varepsilon 2$ for $\beta$ large enough. Now, from (\ref{M-infinito}) we obtain
$$
\|Q+ L_\beta P_j -P\| = \|(Id - L_\beta)Q + L_\beta P_j\|\le \left(1+\frac \varepsilon 2\right)\max\{\|Q\|,\|P_j\|\} =\left(1+\frac \varepsilon 2\right),
$$
and the result follows.
\end{proof}

\begin{remark}\rm
Let $E$ be a Banach space such that $cd(E)>2$ and let $F$ be an infinite dimensional ($M_\infty$)-space. Then, for any degree $n$, with $1\leq n<cd(E)$, $\mathcal{P}_w(^nE,F)$ is a nontrivial $M$-ideal in $\mathcal{P}(^nE,F)$. This is a simple consequence of the above proposition and the fact that $cd(E,F)=1$.
%\textcolor{red}{Lo que sigue es necesario????} \textcolor{blue}{Indeed, by \cite[Theorem 5.1]{HWW}, given a weakly null normalized sequence in $F$, we may find a subsequence which is equivalent to the unit vector basis of $c_0$, say $(y_n)$, so that the closure of $\span\{y_n\}$ is complemented in $F$. Now, by Josefson-Nissenweig Theorem we may take $(\varphi_n)$ in $E^*$ a $w^*$-null sequence of unit vectors and define the continuous linear operator $T\colon E\to F$ by $T(x)=(\varphi_n(x)y_n)$ which would be compact only if $(\varphi_n)$ were norm convergent to 0.  Then, there always exists a non compact operator from $E$ to  the $(M_\infty)$-space $F$ and the assertion follows.}
\end{remark}

The next proposition somehow complements Proposition~\ref{F Moo-space}. It states that if $F$ is an $(M_\infty)$-space, with an additional hypothesis on $E$,
 then $\mathcal{P}_w(^nE,F)$ is  an $M$-ideal in $\mathcal{P}(^nE,F)$ also in the case $n=cd(E)$.

\begin{proposition}\label{F Moo-space-bis}
Let $F$ be an $(M_\infty)$-space and let
$E$ be a Banach space. If $n=cd(E)$ and there exists a
bounded net of compact operators $\{K_\alpha\}_\alpha\subset \mathcal{K}(E)$
satisfying both conditions:
\begin{itemize}
\item $K_\alpha^* (x^*) \to x^*$, for all $x^*\in E^*$.
\item For all $\varepsilon >0$ and all $\alpha_0$ there exists $\alpha >\alpha_0$ such that for every $x\in E$,
$$
\|K_\alpha (x)\|^n + \|x-K_\alpha (x)\|^n \leq (1+\varepsilon)
\|x\|^n,
$$
\end{itemize}
then
$\mathcal{P}_w(^nE,F)$ is an $M$-ideal in $\mathcal{P}(^nE,F)$.
\end{proposition}

\begin{proof}
Let  $P_1, P_2, P_3\in B_{\mathcal{P}_w(^nE,F)}$, $Q\in B_{\mathcal{P}(^nE,F)}$ and $\varepsilon>0$. We will find $P\in \mathcal{P}_w(^nE,F)$ such that the 3-ball property is satisfied. Reasoning as in Theorem~\ref{largo}, we find $\alpha$ and $\beta$ so that $\|P_j-L_\beta P_jK_\alpha\|< \frac\varepsilon 2$. Moreover, $\alpha$ and $\beta$ may be chosen to satisfy at the same time $\|K_\alpha (x)\|^n + \|x-K_\alpha (x)\|^n \leq (1+\tilde\varepsilon)
\|x\|^n$ and $\|L_\beta (y_1)+(Id-L_\beta) (y_2)\| \leq (1+ \tilde\varepsilon)
\max\{\|y_1\|, \|y_2\|\}$ for all $y_1, y_2\in F$; where $\{L_\beta\}_\beta$ is a net in $\mathcal K(F)$, associated to the  $(M_\infty)$-space $F$, and $\tilde\varepsilon$ is such that $(1+\tilde\varepsilon)^2\le 1+\frac\varepsilon 2$.

Let $P$ be the polynomial $P=L_\beta(Q- Q(Id- K_\alpha))$. As in the proof of \cite[Proposition 2.2]{Dim}, we can see that $P$ is weakly continuous on bounded sets at 0. Since $n=cd(E)$ and $P$ is compact, we may appeal to Proposition~\ref{compacto y w0} to derive that $P$ is weakly continuous on bounded sets.

Also we have,
\begin{eqnarray*}
\|Q + P_j  -P\| & \le & \|Q + L_\beta P_j K_\alpha -P\| + \|P_j - L_\beta P_j K_\alpha\|\\
& \leq & \|(Id - L_\beta) Q + L_\beta(P_j K_\alpha + Q(Id - K_\alpha))\| + \frac\varepsilon 2\\
& \leq & (1+\tilde \varepsilon) \sup_{x\in B_E} \max\{\|Q(x)\|,\|P_jK_\alpha (x) + Q(x-K_\alpha (x))\|\} + \frac\varepsilon 2.
\end{eqnarray*}
Now, the hypothesis on $E$ gives us
$$
\|P_jK_\alpha (x) + Q(x-K_\alpha (x))\| \le \|K_\alpha (x)\|^n + \|x-K_\alpha (x)\|^n \leq (1+\tilde \varepsilon),
$$
for all $x\in B_E$, and the result follows.
\end{proof}

\begin{example} \rm Let $E=\ell_p$, with $1<p<\infty$, and let $F$ be an $(M_\infty)$-space. As a consequence of the previous propositions, since $p\le cd(\ell_p)$, $\mathcal{P}_w(^n\ell_p,F)$ is an $M$-ideal in $\mathcal{P}(^n\ell_p,F)$ for all $1\le n \le cd(\ell_p)$.
\end{example}

\bigskip

\section{Polynomials between classical sequence spaces}

This section is devoted to study whether  $\mathcal{P}_w(^nE,F)$ is  an $M$-ideal in $\mathcal{P}(^nE,F)$, for all the values of $n$ between $cd(E,F)$ and $cd(E)$, in the cases $E=\ell_p$ and $F=\ell_q$ or $F$  the Lorentz sequence space $F=d(w,q)$, $1< p, q <\infty$. Recall that given a non increasing sequence $w=(w_j)_j$ of positive real numbers satisfying   $w\in c_0\setminus \ell_1$, the Lorentz sequence space $d(w,q)$ is the space of all sequences $x=(x_j)_j\subset \mathbb K$, such that
$$
\sup_{\sigma} \sum_{j=1}^\infty w_j |x_{\sigma(j)}|^q< \infty,
$$
(where $\sigma$ varies on the set of permutations of $\mathbb N$) endowed with the norm $\|x\|_{d(w,q)}=\displaystyle\sup_{\sigma} \Big(\sum_{j=1}^\infty w_j |x_{\sigma(j)}|^q\Big)^{\frac 1q}$.
We will consider  weights $w=(w_j)_j$ so that $w_1=1$, which implies that the canonical vectors of $d(w,q)$
form a basis of norm 1 elements.

We begin our study with a result about polynomials from a general Banach space $E$ to a Banach space $F$ having a finite dimensional decomposition (FDD) $\{F_n\}_n$. As usual, $\{\pi_m\}_m$ denotes
the sequence of projections associated to the decomposition; that is $\pi_m(y)=\sum_{j=1}^m y_j$ for all $y=\sum_{j=1}^\infty y_j$, with $y_j\in F_j$. Also, we denote by $\pi^m=Id - \pi_m$. When the FDD is unconditional with unconditional constant 1, we have that $\|\pi^m\|\le 1$ and $\|\pi_m + \pi^k\|\le 1$, for all $k\ge m$. In the sequel, we will use, without further mentioning, that for any Banach space $E$ and any $Q\in \mathcal P_w(^n E, F)$, $\|\pi_m Q - Q\| \to 0$, or equivalently, $\|\pi^m Q\| \to 0$, both claims can be derived from the fact that $Q$ is compact.

The following proposition gives conditions under which, if $F$ is a Banach space with 1-unconditional FDD,  $\mathcal{P}_w(^nE,F)$ is not a semi $M$-ideal in $\mathcal{P}(^nE,F)$. This is a polynomial generalization of \cite[Proposition 2]{Oja-91} and our proof is modeled on the proof given in that article. From this, it is obviously inferred that $\mathcal{P}_w(^nE,F)$ is not an $M$-ideal in $\mathcal{P}(^nE,F)$.

\begin{proposition}\label{semi M-ideal}
Let $E$  and $F$ be Banach spaces such that $F$ has an unconditional FDD with unconditional constant equal to 1 and associated projections $\{\pi_m\}_m$. Suppose that there exist polynomials $P\in \mathcal{P}(^nE,F)$ and $Q\in \mathcal{P}_w(^nE,F)$ and numbers $\delta >0$ and $m_0\in \mathbb N$ such that:
\begin{itemize}
\item $0 < \|Q\|\le\|P\|< \delta$,
\item $\|\pi^m P + Q\|\ge \delta$, for all $m\ge m_0$.
\end{itemize}
Then, $\mathcal{P}_w(^nE,F)$ is not a semi $M$-ideal in $\mathcal{P}(^nE,F)$.
\end{proposition}

\begin{proof}
Fix $\varepsilon>0$ so that $\varepsilon< \frac{\delta - \|P\|}{2}$. Since $\|\pi^m Q\| \to 0$, we may assume that  $\|\pi^mQ\|<\frac\varepsilon 3$, for all $m\ge m_0$. Now, fix $m\ge m_0$ and consider the following two closed balls of radius $\|P\|$: $B_1=B(\pi^mP +Q, \|P\|)$ and $B_2=B(\pi^mP -Q, \|P\|)$. Note that $\pi^mP\in B_1\cap B_2$, $Q\in B_1\cap \mathcal{P}_w(^nE,F)$ and $-Q\in B_2\cap \mathcal{P}_w(^nE,F)$.

If $\mathcal{P}_w(^kE,F)$ is a semi $M$-ideal, then for any $r>\|P\|$, the intersection $B(\pi^mP +Q, r)\cap B(\pi^mP -Q, r)\cap \mathcal{P}_w(^nE,F)$ is non void. Take $r=\frac{\|P\|+ \delta}2 -\varepsilon$ and suppose that there exists $R\in B(\pi^mP +Q, r)\cap B(\pi^mP -Q, r)\cap \mathcal{P}_w(^nE,F)$. Since $\|\pi^k R\| \to 0$, we may choose $k\ge m$ such that $\|\pi^k R\|<\varepsilon /3$. To get a contradiction we estimate $\|\pi^kP +Q\|$. Note that

\begin{equation}\label{not-semi-ideal}
2\|\pi^kP +Q\| \le \|\pi^kP +\pi_mQ - \pi_m R\| + \|\pi^kP +\pi_mQ + \pi_m R\| + 2\|\pi^m Q\|.
\end{equation}

\smallskip

From the equality $(\pi_m + \pi^k)(\pi^mP + Q - R)= \pi^kP + \pi_mQ - \pi_mR + \pi^kQ - \pi^k R$, we obtain:

$$
\|\pi^kP +\pi_mQ - \pi_m R\|\le \|\pi_m + \pi^k\| \|\pi^mP + Q - R\| + \|\pi^kQ\|+ \|\pi^k R\|< r + \frac{2\varepsilon}3.
$$
\smallskip

Also, we have that $\|\pi^kP +\pi_mQ + \pi_m R\| =\|\pi^kP  - \pi_mQ - \pi_m R\|$, since $F$ has 1-unconditional finite dimensional decomposition. Proceeding as before, we obtain:

$$
\|\pi^kP -\pi_mQ - \pi_m R\|\le \|\pi_m + \pi^k\| \|\pi^mP - Q - R\| + \frac{2\varepsilon}3 < r + \frac{2\varepsilon}3.
$$

Finally, using~\eqref{not-semi-ideal}, we have that
$$
2\delta \le 2\|\pi^kP +Q\| < 2r + 2\varepsilon  =\|P\| + \delta < 2\delta.
$$
Thus, we conclude that $\mathcal{P}_w(^nE,F)$ is not a semi $M$-ideal in $\mathcal{P}(^nE,F)$.
\end{proof}

Now we can complete the case $E=\ell_p$ and $F=\ell_q$.

\begin{theorem}\label{lp-lq:casos}
 Let $n=cd(\ell_p,\ell_q)$.
\begin{enumerate}
\item[(a)] $\mathcal{P}_w(^n\ell_p,\ell_q)$ is an $M$-ideal in $\mathcal{P}(^n\ell_p,\ell_q)$.
\item[(b)] $\mathcal{P}_w(^k\ell_p,\ell_q)$ is not a semi $M$-ideal in $\mathcal{P}(^k\ell_p,\ell_q)$, for all $k>n$.
\end{enumerate}
\end{theorem}

\begin{proof}
Statement (a) follows from Example~\ref{elepe} and Example~\ref{elepe2}. To prove statement (b) take $k>n$. We will construct polynomials $P\in  \mathcal{P}(^k\ell_p,\ell_q)$ and $Q\in \mathcal{P}_w(^k\ell_p,\ell_q)$ satisfying:   $\|P\|=\|Q\|$ and $\|\pi^mP + Q\|\ge \delta >\|P\|$, for some $\delta >0$, where $\{\pi_m\}_m$ is the sequence of projections associated to the canonical basis of $\ell_q$ and $\pi^m= Id-\pi_m$, for all $m\in \mathbb N$.

We have that $k-1\ge cd(\ell_p,\ell_q)\ge \frac pq$, as shown in Example~\ref{elepe}, so we may define the $k$-homogeneous continuous polynomial $P(x)=e_1^*(x)(x_j^{k-1})_{j\ge 2}$. To compute the norm of $P$, we look, for each $x\in \ell_p$, at the inequality
$$
\|P(x)\|_{\ell_q}=|x_1|\Big(\sum_{j=2}^\infty |x_j|^{(k-1)q}\Big)^{\frac 1q}\le |x_1|\Big(\sum_{j=2}^\infty |x_j|^{p}\Big)^{\frac {k-1}p}.
$$
Then, $\|P\|\le \max\{ab^{k-1}\colon a^p+b^p=1, a, b\ge 0\}=\left[\frac 1k (1-\frac 1k)^{k-1}\right]^{\frac 1p}$. Now, considering
$$
\textstyle \tilde x=\left(\frac1k\right)^{\frac 1p}e_1 + \left(1-\frac 1k\right)^{\frac 1p}e_2,
$$
we obtain a norm one element where $P$ attains the bound $\left[\frac 1k (1-\frac 1k)^{k-1}\right]^{\frac 1p}$.

Let $Q\in  \mathcal{P}_w(^k\ell_p,\ell_q)$ be the polynomial $Q(x)=\|P\|e_1^*(x)^ke_1$. It is clear that $\|P\|=\|Q\|$. Take $m\ge 1$, and $\tilde x=(\frac1k)^{\frac 1p}e_1 + (1-\frac 1k)^{\frac 1p}e_{m+2}$, then $\|\tilde x\|_{\ell_p}=1$ and

$$
\|\pi^mP + Q\|\ge \|(\pi^mP + Q)(\tilde x)\|_{\ell_q}= \textstyle
\left\|(\frac1k)^{\frac 1p}(1-\frac 1k)^{\frac {k-1}p}e_{m+1} + \|P\|(\frac1k)^{\frac kp}e_1\right\|_{\ell_q} =
\|P\|\left( 1 + (\frac 1k)^{\frac{kq}p}\right)^{\frac 1q}.
$$
Then, with $\delta = \left( 1 + (\frac 1k)^{\frac{kq}p}\right)^{\frac 1q}>1$, which is independent of $m$, we obtain the inequality we were looking for. And the theorem is proved.
\end{proof}

Now we focus our attention on spaces of polynomials from $\ell_p$ to $d(w, q)$, $1< p, q <\infty$.  We study whether $\mathcal{P}_w(^k\ell_p,d(w, q))$ is an
$M$-ideal in $\mathcal{P}(^k\ell_p,d(w, q))$ for $k\ge cd(\ell_p,d(w, q))$. To this end we extend to the vector-valued case a couple of results  of \cite{DG} about polynomials from spaces with finite dimensional decompositions.

\begin{lemma}\label{block diagonal} Let $E$ be a Banach space which has an unconditional FDD with associated projections $\{\pi_m\}_m$. For any fixed subsequence $\{m_j\}_j$ of $\mathbb N$, let $\sigma_j=\pi_{m_j} - \pi_{m_{j-1}}$, for all $j$.
Given $P\in \mathcal{P}(^nE,F)$, the application
$$
\tilde P(x)=\sum_{j=1}^\infty P(\sigma_j(x)), \quad for\ all\ x\in E,
$$
defines a continuous $n$-homogeneous polynomial from $E$ to $F$.
\end{lemma}

\begin{proof}
We first show that the series $\sum_{j=1}^\infty P(\sigma_j(x))$ is convergent for every $x\in E$. Indeed, by \cite[Proposition 1.3]{DG}, there exists $C>0$ such that

\begin{eqnarray*}
\Big\|\sum_{j=N}^M P(\sigma_j(x))\Big\| & \le & \sup_{y^*\in B_{F^*}} \sum_{j=N}^M |y^*\circ P(\sigma_j(x))|\\
& = & \sup_{y^*\in B_{F^*}} \sum_{j=N}^M |y^*\circ P(\sigma_j(\pi_{m_M}(x) - \pi_{m_{N-1}}(x)))|\\
& \le & C\|P\| \|\pi_{m_M}(x) - \pi_{m_{N-1}}(x)\|^n,
\end{eqnarray*}
which converges to 0 with $M$ and $N$. Then, $\tilde P(x)$ is well defined and $\|\tilde P\|\le C\|P\|$.
\end{proof}

Recall that whenever a Banach space $E$ has a shrinking FDD, by \cite{AHV}, $\mathcal{P}_w(^nE,F)=\mathcal{P}_{wsc}(^nE,F)$. This allows us to work with sequences instead of nets.

\begin{proposition}\label{todos son wsc0}
Let $E$ be a Banach space with an unconditional FDD and let $F$ be a Banach space. For any $n\in \mathbb N$, the following are equivalent:
\begin{enumerate}
\item[(i)] $\mathcal{P}(^nE,F)=\mathcal{P}_{wsc}(^nE,F)$.
\item[(ii)] $\mathcal{P}(^nE,F)=\mathcal{P}_{wsc0}(^nE,F)$.
\end{enumerate}
\end{proposition}

\begin{proof}
By means of the previous lemma, the scalar valued result given in \cite[Corollary 1.7]{DG} (see also \cite{BR-98}) can be easily modified to obtain this vector valued version.
\end{proof}

In \cite{Oja-91}, Eve Oja studies when $\mathcal K(\ell_p,d(w, q))$ is an $M$-ideal in $\mathcal L(\ell_p,d(w, q))$. In Proposition 1 of that article, she establishes a criterium to ensure that every continuous linear operator is compact. A polynomial version of this result can be stated as follows.

\begin{proposition}\label{oja general}
Let $\{e_j\}_j$ and $\{f_j\}_j$ be sequences in Banach spaces $E$ and $F$, respectively, satisfying:
\begin{itemize}
\item For any semi-normalized weakly null sequence $\{x_m\}_m\subset E$, there exists a subsequence $\{x_{m_j}\}_j$ and an operator $T\in \mathcal L(E)$ such that $T(e_j)=x_{m_j}$, for all $j$.
\item For any semi-normalized weakly null sequence $\{y_m\}_m\subset F$, there exists a subsequence  $\{y_{m_j}\}_j$ and an operator $S\in \mathcal L(F)$ such that $S(y_{m_j})={f_j}$, for all $j$.

\item For any subsequence $\{e_{j_l}\}_l$ of $\{e_j\}_j$, there exists an operator $R\in \mathcal L(E)$ such that $R(e_l)=e_{j_l}$, for all $l$.
\end{itemize}
Take $n< cd(E)$ and suppose that it does not exist a polynomial $P\in \mathcal{P}(^nE,F)$ such that $P(e_j)=f_j$, for every $j$. Then,  $\mathcal{P}(^nE,F)=\mathcal{P}_{wsc0}(^nE,F)$.
\end{proposition}

\begin{proof}
Suppose there exists $P \in \mathcal{P}(^nE,F)$ which is not in
$\mathcal{P}_{wsc0}(^nE,F)$. Then, there exists a weakly null sequence $(x_m)_m$ such that $\|P(x_m)\|>\varepsilon$, for some $\varepsilon>0$ and all $m$. As $n<cd(E)$, $(P(x_m))_m$ is weakly null. Now, we may find a subsequence $(x_{m_j})_j$ and operators $R, T\in \mathcal L(E)$ and $S\in\mathcal L(F)$ satisfying:
$$
e_j\overset{T\circ R}{\longrightarrow} x_{m_j}\overset{P}{\longrightarrow}
P(x_{m_j})\overset{S}{\longrightarrow} f_j,$$
which is a contradiction since $S\circ P \circ T\circ R$ belongs to $\mathcal{P}(^nE,F)$.
\end{proof}

\begin{remark}\label{Pek=fk}\rm
If the Banach space $E$ has an unconditional basis $\{e_j\}_j$ with coordinate functionals $\{e_j^*\}_j$ and $\{f_j\}_j$ is a sequence in the Banach space F, we derive from Lemma \ref{block diagonal} that the existence of a polynomial $P\in \mathcal{P}(^nE,F)$ such that $P(e_j)=f_j$, for all $j$, is equivalent to the existence of the polynomial $\widetilde{P}\in \mathcal{P}(^nE,F)$ given by
$$
\widetilde{P}(x)=\sum_{j=1}^\infty  \left(e_j^*(x)\right)^n\, f_j,\qquad\textrm{for all }x\in E.\quad
$$
When $E$ and $F$ are Banach sequence spaces with canonical bases $\{e_j\}_j$ and $\{f_j\}_j$ respectively, we write the polynomial above as $\widetilde{P}(x)=(x_j^n)_j$. \end{remark}

Let $1<p,q<\infty$. To study whether $\mathcal{P}_w(^n\ell_p,d(w, q))$ is an $M$-ideal in $\mathcal{P}(^n\ell_p,d(w, q))$ for $n\geq cd(\ell_p,d(w, q))$, we need first to establish the value of the critical degree, $cd(\ell_p,d(w, q))$. To this end and in view of the previous remark and proposition, the point is to determine the values of $n$, $p$ and $q$ such that the polynomial $x\mapsto (x_j^n)_j$, from $\ell_p$ to $d(w,q)$, is well defined. For $1\le r <\infty$ we use the standard notation $s=r^*$ to denote de conjugate number of $r$: $\frac 1r + \frac 1s =1$.

\begin{proposition}\label{ej_ellp_lorentz}
The polynomial $P(x)=(x_j^n)_j$ belongs to $\mathcal{P}(^n\ell_p,d(w, q))$ if and only if one of the following two conditions holds:
\begin{enumerate}
\item[(a)] $n\ge \frac{p}{q}$. In this case, $\|P\|=1$.
\item[(b)] $n<\frac{p}{q}$ and $w\in \ell_{s}$, for $s=(\frac{p}{nq})^*$. In this case, $\|P\|= \|w\|_{\ell_{s}}^{\frac 1q}$.
\end{enumerate}
\end{proposition}

\begin{proof}
Let $(e_j)_j$ and $(f_j)_j$ be the canonical bases of $\ell_p$ and
$d(w, q)$, respectively. Suppose that $n \ge \frac pq$, as $\|w\|_\infty=1$, we have
$$
\|P(x)\|_{d(w, q)}=\sup_\sigma \Big( \sum_{j=1}^\infty w_j|x_{\sigma(j)}|^{nq}\Big)^{\frac 1q} \le \|x\|_{\ell_p}^n.
$$
Then, $P$ is a well defined polynomial with norm less than or equal to 1. Also, $P(e_j)=f_j$ implies $\|P\|=1$.

Now, suppose that $n<\frac pq$ and $w\in \ell_{s}$, with  $s=(\frac{p}{nq})^*$. Put $W=\|w\|_{\ell_{s}}$, by H\"older inequality, we have
$$
\|P(x)\|_{d(w, q)}=\sup_\sigma \Big( \sum_{j=1}^\infty w_j|x_{\sigma(j)}|^{nq}\Big)^{\frac 1q} \le W^{\frac 1q} \|x\|_{\ell_p}^n.
$$
Thus, $\|P\|\le W^{\frac 1q}$ and considering $\tilde x=W^{-\frac sp}(w_j^{\frac sp})_j\in S_{\ell_p}$, we obtain that $\|P\|= W^{\frac 1q}=\|w\|_{\ell_{s}}^{\frac 1q}$.

Finally, suppose that $n<\frac pq$ and $w\not\in \ell_{s}$. Then, there exists $(b_j)_j\in\ell_{\frac p{nq}}$ with $b_1\ge b_2\ge b_3\ge \cdots\ge 0$ such that the series $\sum_{j=1}^\infty w_jb_j$ does not converge. Taking $\tilde x\in \ell_p$, $\tilde x= (b_j^{\frac 1{nq}})_j$ we have that $P(\tilde x)=(b_j^{\frac 1{q}})_j\not\in d(w,q)$. Now, the proof is complete.
\end{proof}

\begin{proposition}
$\mathcal{P}_w(^n\ell_p,d(w, q))=\mathcal{P}(^n\ell_p,d(w, q))$ if and only if $n<\frac{p}{q}$ and $w\not\in \ell_{s}$, with  $s=(\frac{p}{nq})^*$.
\end{proposition}

\begin{proof}
 By the previous proposition, whenever $n\ge \frac{p}{q}$ or
$n<\frac{p}{q}$ and $w\in \ell_{s}$, $s=(\frac{p}{nq})^*$, the polynomial  $P(x)=(x_j^n)_j$ belongs to $\mathcal{P}(^n\ell_p,d(w, q))$ and fails to be weakly continuous on bounded sets.

For the converse, we have that $E=\ell_p$ and $F=d(w,q)$ satisfy the three conditions of Proposition~\ref{oja general}, with $(e_j)_j$ and $(f_j)_j$ the respective canonical basis of $\ell_p$ and $d(w, q)$, see \cite[Corollary 2]{Oja-91}. Also, by Remark~\ref{Pek=fk} and Proposition~\ref{ej_ellp_lorentz}, we have that it does not exist $P\in \mathcal{P}(^n\ell_p,d(w, q))$ such that $P(e_j)=f_j$, for every $j$. Finally, as $n<\frac pq\le p\le cd(\ell_p)$, all the hypothesis of Proposition~\ref{oja general} are fulfilled. Therefore, $\mathcal{P}(^n\ell_p,d(w, q))=\mathcal{P}_{wsc0}(^n\ell_p,d(w, q))$. Now, by Proposition~\ref{todos son wsc0},
$\mathcal{P}(^n\ell_p,d(w, q))=\mathcal{P}_{wsc}(^n\ell_p,d(w, q))$ and
the result follows from \cite{AHV}, since weakly sequentially continuous polynomials and weakly continuous polynomials on bounded sets coincide on $\ell_p$.
\end{proof}

Let  $n=cd(\ell_p,d(w, q))$. Taking into account that for every  $k<n$, any polynomial in $\mathcal P(^k\ell_p,d(w, q))$ is weakly continuous on bounded sets, from the last proposition we derive that there are two possible values for $n$:

\begin{enumerate}
\item[(I)] $\frac{p}{q}\le n <\frac{p}{q}+1$ and $w\not\in \ell_{\left(\frac{p}{(n-1)q}\right)^*}$, or
\item[(II)] $n<\frac{p}{q}$ and $w\in \ell_{\left(\frac{p}{nq}\right)^*}\setminus \ell_{\left(\frac{p}{(n-1)q}\right)^*}$.
\end{enumerate}
\smallskip

\begin{theorem}
Let $n=cd(\ell_p,d(w, q))$.
\begin{enumerate}
\item[(a)] If $n$ and $w$ satisfy condition {\rm(I)} above, then
\begin{itemize}
\item $\mathcal{P}_w(^n\ell_p,d(w, q))$ is an $M$-ideal in $\mathcal{P}(^n\ell_p,d(w, q))$, and
\item $\mathcal{P}_w(^k\ell_p,d(w, q))$ is not a semi $M$-ideal in $\mathcal{P}(^k\ell_p,d(w, q))$, for all $k>n$.
\end{itemize}
\item[(b)] If $n$ and $w$ satisfy condition {\rm(II)} above, then $\mathcal{P}_w(^k\ell_p,d(w, q))$ is not a semi $M$-ideal in $\mathcal{P}(^k\ell_p,d(w, q))$, for all $k\ge n$.
\end{enumerate}
\end{theorem}

\begin{proof}
Suppose $n$ and  $w$ satisfy condition (I) above. Then, $n=cd(\ell_p,d(w, q))\ge \frac pq$ and $cd(\ell_p)$ is the integer number satisfying $p\le cd(\ell_p) < p+1$. If $n< cd(\ell_p)$, the hypothesis of Theorem~\ref{largo} are fulfilled. If $n=cd(\ell_p)$ we may apply Proposition~\ref{condition}. In both cases the conclusion follows.

Now, take $k>cd(\ell_p,d(w, q))$. According to Proposition~\ref{semi M-ideal}, the result is proven if we find polynomials $P\in \mathcal{P}(^k\ell_p,d(w, q))$ and $Q\in \mathcal{P}_w(^k\ell_p,d(w, q))$ such that there exists $\delta>0$ with $\|P\|=\|Q\|$ and $\|\pi^m P + Q\|\ge \delta >\|P\|$, for all $m$.

By Proposition~\ref{ej_ellp_lorentz}, as $k-1\ge \frac pq$, the mapping $R(x)=(x_{j}^{k-1})_{j\ge 2}$ is a well defined norm one polynomial. Then, $P(x)= e_1^*(x)R(x)$ belongs to  $\mathcal{P}(^k\ell_p,d(w, q))$. In order to compute its norm, take $x$ so that $\|x\|_{\ell_p}=1$,
$$\textstyle
\|P(x)\|_{d(w, q)}=|x_1|\|R(x)\|_{d(w, q)}\le |x_1|\|(x_j)_{j\ge 2}\|^{k-1}_{\ell_p}\le \Big(\frac 1k\Big)^{\frac 1p}\Big(1-\frac 1k\Big)^{\frac{k-1}p},
$$
where the last inequality was shown in the proof of Theorem~\ref{lp-lq:casos}. Now, with $\tilde x= (\frac 1k)^{\frac 1p}e_1+ (1-\frac 1k)^{\frac 1p}e_2\in S_{\ell_p}$ we have that  $P(\tilde x)=(\frac 1k)^{\frac 1p}(1-\frac 1k)^{\frac {k-1}p}e_1$, whence $\|P\|= (\frac 1k)^{\frac 1p}(1-\frac 1k)^{\frac{k-1}p}$.

Let $Q$ be the weakly continuous on bounded sets polynomial given by $Q=\|P\|(e_1^*)^k\cdot e_1$. Then, $\|Q\|=\|P\|$ and  $\tilde x=(\frac 1k)^{\frac 1p}e_1 + (1-\frac 1k)^{\frac 1p}e_{m+2}$, for $m\ge 1$, is a norm one vector so that
$$
\|(\pi^m P + Q)(\tilde x)\|_{d(w,q)}= \|P\| \Big\| e_{m+1} + (\textstyle{\frac 1k})^{\frac kp}e_1 \Big\|_{d(w,q)}=\|P\|\Big(1 + w_2(\frac 1k)^{\frac{kq}p}\Big)^{\frac 1q} > \|P\|,
$$
which completes the proof of (i).

To prove (ii), take $n$ and $w$ satisfying condition (II)  and take $k\ge n$. Let us denote $s=(\frac{p}{nq})^*=\frac p{p-nq}$ and $W=\|w\|_{\ell_{s}}$. By Proposition \ref{ej_ellp_lorentz} (b), the $n$-homogeneous polynomial $R(x)=(x_j^n)_j$ satisfies
$$
\|R(\tilde x)\|_{d(w,q)}=\|R\|= W^{\frac 1q},\qquad\textrm{where }\tilde x=W^{-\frac sp}(w_j^\frac sp)_j.
$$
Observe that $x^*=(w_j^\frac s{p^*})_j$ belongs to $\ell_{p^*}$ and, as a continuous functional, it also attains its norm at $\tilde x$:
$$
x^*(\tilde x)=\|x^*\|=W^{\frac s {p^*}}.
$$
Now we are ready to construct two polynomials $P$ and $Q$ fulfilling the statement of Proposition~\ref{semi M-ideal}. Let $P\in \mathcal{P}(^k\ell_p,d(w, q))$  and $Q\in \mathcal{P}_w(^k\ell_p,d(w, q))$ be given by
$$
P(x)=x^*(x)^{k-n}R(x)\qquad\textrm{and }\qquad Q(x)=W^{\frac 1q -\frac{sn}{p^*}}x^*(x)^k e_1.
$$
It is easy to see that
$$
\|P(\tilde x)\|_{d(w,q)}=\|P\|=W^r=\|Q(\tilde x)\|_{d(w,q)}=\|Q\|,\qquad\textrm{where }r=\frac{s(k-n)}{p^*}+\frac 1q.
$$
Finally,
\begin{eqnarray*}
\|\pi^m P + Q\| & \ge & \|(\pi^m P + Q)(\tilde x)\|_{d(w,q)}=\left\|W^r e_1 + W^{\frac{s(k-n)}{p^*}}\sum_{j=m+1}^\infty \tilde x_j^n e_j\right\|_{d(w,q)} \\
& = & W^{r}\left[1 + W^{-1} \sum_{j=2}^\infty w_j |\tilde x_{m-1+j}|^{nq}\right]^{\frac 1q} > W^r=\|P\|.
\end{eqnarray*}
This completes the proof of the theorem.
\end{proof}

\section{Polynomial property $(M)$}

Property $(M)$ was introduced by Kalton in \cite{K}. It is a geometric property relating the norm of the traslation by a weakly null net of any two elements of the space. Namely, a Banach space $X$ has property $(M)$ if for any $x, \tilde x\in X$ such that $\|x\|\le \|\tilde x\|$, and any bounded weakly null net $(x_\alpha)_\alpha$ in $X$, it holds that $\limsup \|x + x_\alpha\|\le \limsup \|\tilde x + x_\alpha\|$. An operator version of this property was given in  \cite{KW}. Later on, in \cite{Dim}, it is extended to the scalar-valued polynomial context. In all these cases, these properties have incidence in the correspondent $M$-ideal problems. To study $M$-structures in spaces of vector-valued polynomials, we consider a suitable property $(M)$, which is the result of a natural combination of the definitions given for  operators and scalar-valued polynomials. Before going on, let us state the vector-valued versions of \cite[Lemma 3.1 and Theorem 3.2]{Dim}.

\begin{lemma}
If $\mathcal{P}_w(^nE,F)$ is an $M$-ideal in $\mathcal{P}(^nE,F)$ then,
for each $P\in \mathcal{P}(^nE,F)$ there exists a bounded net
$\{P_\alpha\}_\alpha\subset \mathcal{P}_w(^nE,F)$ such that
${P}_\alpha(x)\to {P}(x)$, for all $x\in E$.
\end{lemma}

\begin{proof}
 Fix $P\in \mathcal{P}(^nE,F)$. By \cite[Remark I.1.13]{HWW}, we may consider $\{Q_\alpha\}_\alpha$ a bounded net in $\mathcal{P}_w(^nE, F)$ such that $Q_\alpha \to P$ in the topology $\sigma\Big(\mathcal{P}(^nE, F),\mathcal{P}_w(^nE, F)^*\Big)$.

Since $e_x\otimes y^*$ belongs to $\mathcal{P}_w(^nE, F)^*$,
$y^*(Q_\alpha (x))=\langle e_x\otimes y^*, Q_\alpha \rangle \to \langle e_x\otimes y^*, P \rangle=y^*(P(x))$, for all $x\in E$ and all $y^*\in F^*$. This says that $Q_\alpha (x)\overset{w}{\rightarrow} P(x)$, for all $x\in E$, which can be described, in analogy to the operator setting, as $Q_\alpha \to P$ in the {\it WPT}, the ``weak polynomial topology''.

We can also consider on $\mathcal{P}(^nE,F)$ the ``strong polynomial topology'', {\it SPT}, naturally meaning pointwise convergence of nets. Both topologies, the {\it WPT} and the {\it SPT}, are locally convex and
have the same continuous functionals (the proof of \cite[Theorem VI.1.4]{DS1} works also for polynomials). Thus, as in the linear case, we derive that the closure of any convex set in the strong polynomial topology coincides with its closure in the weak polynomial topology.

Then, we may find $P_\alpha$, a convex combination of $Q_\alpha$, converging pointwise to $P$.
\end{proof}

As a consequence of \cite[Proposition 2.3]{W} and the previous
lemma, we have the following result which can be proved analogously
to \cite[Theorem 3.1]{W}:

\begin{theorem}\label{equivalencias}
Let $E$ and $F$ be Banach spaces. The following are equivalent:
\begin{enumerate}
\item[(i)] $\mathcal{P}_w(^nE,F)$ is an
$M$-ideal in $\mathcal{P}(^nE,F)$.
\item[(ii)] For all $P\in \mathcal{P}(^nE,F)$ there exists a net
$\{P_\alpha\}_\alpha\subset \mathcal{P}_w(^nE,F)$ such that
$P_\alpha(x)\to P(x)$, for all $x\in E$
and
$$
\limsup \|Q+P-P_\alpha\|\leq \max\{\|Q\|, \|Q\|_{es} + \|P\|\},\quad
\textrm{for all } Q\in\mathcal{P}(^nE,F).
$$
\item[(iii)] For all $P\in \mathcal{P}(^nE,F)$ there exists a net
$\{P_\alpha\}_\alpha\subset \mathcal{P}_w(^nE,F)$ such that
$P_\alpha(x)\to P(x)$, for all $x\in E$
and
$$
\limsup \|Q+P-P_\alpha\|\leq \max\{\|Q\|, \|P\|\},\quad
\textrm{for all } Q\in\mathcal{P}_w(^nE,F).
$$
\end{enumerate}
\end{theorem}

Now we state the property $(M)$ for a vector-valued polynomial.

\begin{definition}
Let $P\in\mathcal{P}(^nE,F)$ with $\|P\|\leq 1$. We say that $P$ has
{\bf property $(M)$} if for all  $u\in E$, $v\in F$
with $\|v\|\leq\|u\|^n$ and for every bounded weakly null net
$\{x_\alpha\}_\alpha\subset E$, it holds that
$$
\limsup_\alpha \|v+P(x_\alpha)\|\leq \limsup_\alpha
\|u+x_\alpha\|^n.
$$
\end{definition}

Note that every $P\in \mathcal{P}_w(^nE,F)$ with $\|P\|\le 1$ has property $(M)$.
Analogously to \cite[Lemma 6.2]{KW}, we can prove:

\begin{lemma}\label{redes}
Let $P\in\mathcal{P}(^nE,F)$ with $\|P\|\leq 1$. If $P$ has property
$(M)$ then for all nets $\{u_\alpha\}_\alpha$ and $\{v_\alpha\}_\alpha$ contained in
compact sets of $E$ and $F$ respectively,
with $\|v_\alpha\|\leq\|u_\alpha\|^n$ and for every bounded
weakly null net $\{x_\alpha\}_\alpha\subset E$, it holds that
$$
\limsup_\alpha \|v_\alpha+P(x_\alpha)\|\leq \limsup_\alpha
\|u_\alpha+x_\alpha\|^n.
$$
\end{lemma}

\begin{definition}
We say that a pair of Banach spaces $(E,F)$ has the {\bf $n$-polynomial property
$(M)$} if every $P\in\mathcal{P}(^nE,F)$ with $\|P\|\leq 1$ has
property $(M)$.
\end{definition}

The next two results can be proved mimicking the proofs of  Proposition~3.7 and Theorem~3.9 of \cite{Dim}.

\begin{proposition}
If $\mathcal{P}_w(^nE,F)$ is an $M$-ideal in $\mathcal{P}(^nE,F)$ and $n=cd(E,F)$ then
$(E,F)$ has the $n$-polynomial property $(M)$.
\end{proposition}

\begin{theorem}\label{n-prop-M}
Let $E$ and $F$ be Banach spaces and suppose that there exists a net
of compact operators $\{K_\alpha\}_\alpha\in \mathcal{K}(E)$
satisfying the following two conditions:
\begin{itemize}
\item $K_\alpha (x)\to x$, for all $x\in E$ and $K_\alpha^* (x^*) \to x^*$, for all $x^*\in E^*$.
\item $\|Id - 2 K_\alpha\|\underset{\alpha}{\longrightarrow} 1$.
\end{itemize}
Suppose also that $n=cd(E,F)$. Then,
$\mathcal{P}_w(^nE,F)$ is an $M$-ideal in $\mathcal{P}(^nE,F)$ if and
only if $(E,F)$ has the $n$-polynomial property $(M)$.
\end{theorem}

Sometimes it is possible to infer $M$-structures in the space of linear continuous operators from the existence of geometric structures on the underlying space. For instance, it is proved in \cite[Theorem VI.4.17]{HWW} that $\mathcal K(E)$ is an $M$-ideal in $\mathcal L(E)$ if and only if $E$ has property $(M)$ and satisfies both conditions of the theorem above. A similar result \cite[Theorem~3.9]{Dim} is obtained in the scalar-valued polynomial setting for $n=cd(E)$ using the polynomial property (M). The following proposition (which is the vector-valued polynomial version of \cite[Lemma VI.4.14]{HWW} and \cite[Proposition 3.10]{Dim}) paves the way to connect the linear $M$-structure with $M$-ideals in vector valued polynomial spaces.

\begin{proposition}\label{M}
Let $E$ and $F$ be  Banach spaces and
$n=cd(E,F)<cd(E)$. If $E$ and $F$ have the property
$(M)$, then $(E,F)$ has the $n$-polynomial property $(M)$.
\end{proposition}

\begin{proof}
 Let $P\in\mathcal{P}(^nE,F)$ with $\|P\|= 1$. Fix
$u\in E$, $v\in F$ with $\|v\|\leq\|u\|^n$ and a bounded weakly null net
$\{x_\alpha\}_\alpha\subset E$. We want to
prove that
$$
\limsup_\alpha \| v +P(x_\alpha)\|\leq \limsup_\alpha
\|u+x_\alpha\|^n.
$$
Given $\varepsilon >0$, take $x\in S_E$ such that $\|P(x)\|>1-\varepsilon$ and $\tilde x=\|v\|^{\frac 1n} x$.  Then, $(1-\varepsilon)\|v\| < \|P(\tilde x)\|\le \|\tilde x\|\le \|u\|$.
As $n<cd(E)$, every scalar valued polynomial in $\mathcal P(^n E)$ is weakly continuous on bounded sets. Then, $P$ is weak-to-weak continuous and $P(x_\alpha) \overset{w}{\rightarrow} 0$. Therefore, since $F$ has property $(M)$,

\begin{eqnarray*}
\limsup_\alpha \|(1-\varepsilon)v + P(x_\alpha)\|&\leq & \limsup_\alpha
\|P (\tilde x) +P(x_\alpha)\|\\
&=&\limsup_\alpha \|P(\tilde x +x_\alpha)\| \\
&\leq & \limsup_\alpha \|\tilde x +x_\alpha\|^n \\
&\leq & \limsup_\alpha \|u+x_\alpha\|^n,
\end{eqnarray*}
where the last inequality holds since $E$ has property $(M)$. Now, letting $\varepsilon\to 0$ we obtain the desired inequality.

If $\|P\|<1$, the result follows from the previous case through the following convex combination
$$v +P(x_\alpha)=\frac{1+\|P\|}2\left(v + \frac{P}{\|P\|}(x_\alpha)\right) + \frac{1-\|P\|}2\left(v - \frac{P}{\|P\|}(x_\alpha)\right).$$
\end{proof}

Now we can lift $M$-structures from the linear to the vector-valued polynomial context. This is done for the particular case of $n$-homogeneous polynomials when $n$ is the critical degree of the pair $(E, F)$ and it is strictly less than the critical degree of the domain space $E$. We do not know if the result remains true even for the case $n=cd(E,F)=cd(E)$.

\begin{corollary}\label{lineal}
Let $E$ and $F$ be  Banach spaces and
$n=cd(E,F)<cd(E)$. If $\mathcal{K}(E)$ is
an $M$-ideal in $\mathcal{L}(E)$  and $F$ has property $(M)$, then $\mathcal{P}_w(^nE,F)$ is an
$M$-ideal in $\mathcal{P}(^nE,F)$.
\end{corollary}

\begin{proof}
If $\mathcal{K}(E)$ is an $M$-ideal in $\mathcal{L}(E)$, appealing to \cite[Theorem VI.4.17]{HWW}, $E$ has property $(M)$ and we may find $\{K_\alpha\}_\alpha\subset \mathcal{K}(E)$ a net of compact operators satisfying both conditions of Theorem~\ref{n-prop-M}. By Proposition~\ref{M}, $(E,F)$ has the $n$-polynomial property $(M)$. Now, we may apply  Theorem~\ref{n-prop-M} to derive the result.
\end{proof}

We finish this section applying the previous result to give some examples of $M$-ideals of polynomials between Bergman and $\ell_p$ spaces.

\begin{example}\rm
The Bergman space $B_p$ is the
space of all holomorphic functions in $L_p(\mathbb{D},dxdy)$, where $\mathbb{D}$ is the complex disc. If
$1<p<\infty$, $B_p$ is isomorphic to $\ell_p$ \cite[Theorem
III.A.11]{Woj} and so, for $1<p,q<\infty$,
$$
cd(\ell_p,\ell_q)=cd(\ell_p,B_q)=cd(B_p,\ell_q)=cd(B_p,B_q).
$$
Since, by \cite[Corollary 4.8]{KW}, $\mathcal{K}(B_p)$
is an $M$-ideal in $\mathcal{L}(B_p)$, we obtain from Corollary \ref{lineal}, that, if $n=cd(\ell_p,\ell_q)<cd(\ell_p)$, then:
\begin{itemize}
\item $\mathcal{P}_w(^n\ell_p,B_q)$
is an $M$-ideal in $\mathcal{P}(^n\ell_p,B_q)$.
\item $\mathcal{P}_w(^nB_p,\ell_q)$
is an $M$-ideal in $\mathcal{P}(^nB_p,\ell_q)$.
\item $\mathcal{P}_w(^nB_p,B_q)$
is an $M$-ideal in $\mathcal{P}(^nB_p,B_q)$.
\end{itemize}
\end{example}

\end{document}